\documentclass[11pt,leqno]{article}

\usepackage{amsmath,amssymb,amsfonts,amsthm}

\usepackage{graphicx}
\usepackage{extarrows, textcomp}

\usepackage[pdftex,plainpages=false,hypertexnames=false,pdfpagelabels]{hyperref}
\usepackage{tikz}
\usetikzlibrary{shapes,decorations,decorations.pathreplacing,fit,calc,patterns}

\usepackage{enumerate}
\usepackage{microtype}

\usepackage{color}

\usepackage{etoolbox}

\usepackage{xcolor}
\definecolor{dark-red}{rgb}{0.6,0.15,0.15}
\definecolor{dark-blue}{rgb}{0.15,0.15,0.55}
\definecolor{medium-blue}{rgb}{0,0,0.65}
\hypersetup{
    colorlinks, linkcolor={dark-red},
    citecolor={dark-blue}, urlcolor={medium-blue}
}

\def\nn#1{{{\color[rgb]{.2,.5,.6} \small [[#1]]}}}

\long\def\noop#1{}

\def\semicolon{;}
\def\applytolist#1{
    \expandafter\def\csname multi#1\endcsname##1{
        \def\multiack{##1}\ifx\multiack\semicolon
            \def\next{\relax}
        \else
            \csname #1\endcsname{##1}
            \def\next{\csname multi#1\endcsname}
        \fi
        \next}
    \csname multi#1\endcsname}

\def\calc#1{\expandafter\def\csname c#1\endcsname{{\mathcal #1}}}
\applytolist{calc}QWERTYUIOPLKJHGFDSAZXCVBNM;

\def\declaremathop#1{\expandafter\DeclareMathOperator\csname #1\endcsname{#1}}
\applytolist{declaremathop}{gl}{mor}{Maps}{Diff}{Homeo}{Rep}{End}{Ind}{Res}{Hom}{Stab}{ev}{id}{tr}{GDim}{GD}{Spin}{Pin}{Link};

\theoremstyle{plain}
\newtheorem{prop}{Proposition}[subsection]

\newtheorem{thm}[prop]{Theorem}
\newtheorem{lem}[prop]{Lemma}

\newtheorem*{cor*}{Corollary}

\newtheorem*{defn*}{Definition}             

\numberwithin{figure}{subsection}
\numberwithin{equation}{subsection}
\AtBeginEnvironment{figure}{\stepcounter{equation}}
\AtBeginEnvironment{equation}{\stepcounter{figure}}

\makeatletter
\let\c@prop\c@figure
\makeatother

\def\eqar#1{\begin{eqnarray*}#1\end{eqnarray*}}

\def\ip#1#2{\langle #1, #2 \rangle}
\def\ipp#1{\langle #1, #1 \rangle}

\newcommand*{\deq}{\mathrel{\vcenter{\baselineskip0.5ex \lineskiplimit0pt
                     \hbox{\scriptsize.}\hbox{\scriptsize.}}}%
                     =}

\def\du{\sqcup}
\def\bd{\partial}
\def\sub{\subset}

\def\setmin{\setminus}
\def\ol{\overline}

\def\eset{\phi}

\def\ot{\otimes}
\def\tunit{\mathbf{1}}
\def\gdim{\GD}
\def\trev#1{\ol{#1}}
\def\inv{^{-1}}
\def\rev#1{\ol{#1}}
\def\hb#1{{h(#1)}}

\def\z{\mathbb{Z}}

\def\c{\mathbb{C}}

\def\k{\Bbbk}

\def\cl{\c l}

\graphicspath{ {figs/} }
\def\kfigb#1#2#3#4{
	\begin{figure}
	\centering
	\includegraphics[width=#4]{#2}
	\caption{#3}
	\label{#1}
	\end{figure}
}

\usepackage[backend=biber,style=alphabetic,doi=false,isbn=false,url=false,minnames=6,maxnames=6]{biblatex}
\renewbibmacro{in:}{%
  \ifentrytype{article}{}{\printtext{\bibstring{in}\intitlepunct}}}
\renewbibmacro*{volume+number+eid}{%
  \printtext{vol.}
  \printfield{volume}%
  \setunit*{\addnbspace}
  \printfield{number}%
  \setunit{\addcomma\space}%
  \printfield{eid}}
\DeclareFieldFormat[article]{number}{\mkbibparens{#1}}
\usepackage{silence}
\newcommand{\doi}[1]{\href{https://dx.doi.org/#1}{{\small DOI:#1}}}

\newcommand{\googlebooks}[1]{(preview at \href{https://books.google.com/books?id=#1}{google books})}

\newcommand{\numdam}[1]{}

\title{A universal state sum}
\author{Kevin~Walker\\Microsoft Station Q\\kevin@canyon23.net}

\begin{document}

\maketitle

\begin{abstract}
We define a universal state sum construction which specializes to most previously known
state sums (Turaev-Viro, Dijkgraaf-Witten, Crane-Yetter, Douglas-Reutter, Witten-Reshetikhin-Turaev surgery formula, Brown-Arf).
The input data for the state sum is an $n$-category satisfying various conditions, 
including finiteness, semisimplicity and $n$-pivotality.
From this $n$-category one constructs an $n{+}1$-dimensional TQFT, and applying the TQFT
gluing rules to a handle decomposition of an $n{+}1$-manifold produces the state sum.
\end{abstract}

\tableofcontents

\section{Introduction}
\label{intro-sect}

Given the right sort of $n$-category, one can construct a fully extended $n{+}1$-dimensional TQFT.
The 0- through $n$-dimensional parts of this TQFT can be defined
without choosing any combinatorial description of manifolds.
(See Appendix \ref{s-nep}.)
The top, $n{+}1$-dimensional part of the TQFT (the path integral)
is defined in terms of a handle decomposition of an $n{+}1$ manifold.
(See Section \ref{s-np1}.)

The computation of the path integral in terms of a handle decomposition
gives an algorithm for computing the path integral.
This algorithm was described in \cite{W2006}.
It was also described in \cite{W2006} how to derive the Turaev-Viro and Crane-Yetter state sums from the algorithm.
The main new thing in this paper is transforming the algorithm into a more concise formula
which I'll call the universal state sum.

The universal state sum specializes straight-forwardly to a long list of previously known state sums
(Turaev-Viro, Dijkgraaf-Witten, Crane-Yetter, Douglas-Reutter, 
Witten-Reshetikhin-Turaev (thought of as computing a relative  Crane-Yetter invariant), Brown-Arf,
Turaev ``shadow" sum).
It also produces some new examples (e.g.\ a Turaev-Viro-like state sum for super pivotal categories).

Most papers in the state sum literature work in terms of triangulations rather than handle decompositions.
Any triangulation can be turned into a handle decomposition by thickening the cells of the 
generic cell decomposition which is Poincar\'e dual to the triangulation.
(See Figure \ref{c1}.)
\kfigb{c1}{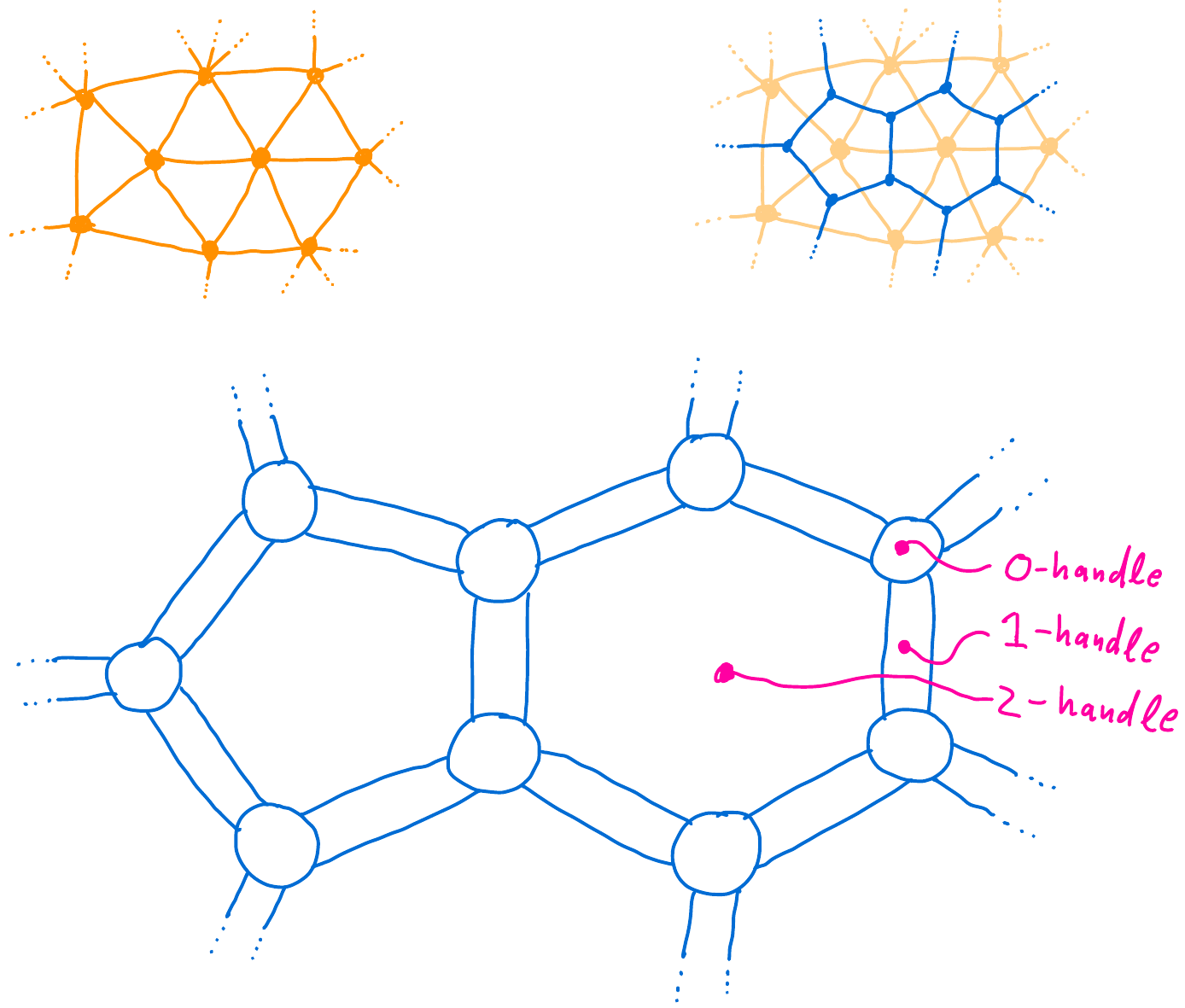}{Upper left: orange triangulation ($n=2$ case); upper right: blue dual cell decomposition; 
bottom: handle decomposition from thickened cell decomposition.}{6in}
I prefer working with handle decompositions (or equivalently, general cell decompositions), because 
(1) they are more general (the Witten-Reshetikhin-Turaev and Turaev-shadow sums cannot be described in terms of triangulations), and
(2) the combinatorial ``moves" for relating handle decompositions (handle slides and cancellations) are 
much simpler that the moves for relating triangulations (Pachner moves).

The class of manifolds on which the TQFT and state sum is defined 
can be oriented, unoriented, Spin, or $\mbox{Pin}_\pm$; ``$H$-manifolds" for short (where $H$ is $\mbox{Spin}(n+1)$ or $SO(n+1)$ or $O(n+1)$ or $\ldots$).
The only requirements are that the manifolds admit handle decompositions, and that any two handle decompositions
of a manifold are related by handle slides and handle cancellations.
Because utilize gluing manifolds with corners, it is most convenient to work with PL manifolds.
Smooth manifolds will also work, if we adopt the convention that every smooth $k$-manifold
is implicitly equipped with a germ of $n{+}1$-dimensional smooth neighborhoods.

(The Spin and Pin cases require a certain amount of fussiness to do precisely 
(see \cite{ALW} for the $2{+}\epsilon$-dimensional Spin case of this fussiness),
so the details for those cases will appear in a sequel to this paper.
This paper focuses on the oriented and unoriented cases.
The  original plan was to put all cases in a single paper,
and a few stray references to the Spin and Pin cases remain in this paper.)

The higher category arguments used in this paper are mostly string-diagram-theoretic, and 
assume that the input $n$-category satisfies strong duality conditions (strict-pivotal in the oriented $n=2$ case, ``$H$-pivotal" in general) and is semi-simple.
Any model of $H$-pivotal higher categories which supports string diagrams should suffice.
(See \ref{n-cat} for a partial list of such models.)

The inductive path integral construction of Section \ref{s-np1} is the core of the argument.
The argument presented here is very nearly the same as the one presented in \cite{W2006}.
The proof relies heavily on semisimplicity assumptions and will likely strike algebraists and category theorists
as a bit clunky.
In recent joint work with David Reutter \cite{RWnss}, the inductive construction of the path integral
has been generalized to non-semisimple contexts using less clunky techniques.
Algebraists and category theorists will likely prefer the new, more general proof.
But the older, less fancy proof presented here might appeal to more to low-dimensional topologists mainly interested 
in the semi-simple case.

This work was initiated and mostly completed in Spring 2020 at the Mathematical Sciences Research Institute, 
and I gratefully acknowledge the excellent working environment MSRI provided.

\noop{
This paper was originally intended to cover the Spin and Pin cases of $H$ and as well as the oriented and unoriented cases.
In order to avoid lengthy and distracting forays into Spin and Pin fussiness, those cases have been postponed
and will appear in a sequel paper.
A few stray references to the Spin and Pin cases remain in this paper.
}

\noop{
\bigskip
For readers inclined to skim, here's a table of notation.

\medskip

\def\kwstr{\rule[-3mm]{0mm}{9mm}}
\begin{tabular}{| l | p{5in} |}
 \hline
\kwstr $C$ & an $n$-category as described in \ref{n-cat}; $\k$-linear, $H$-pivotal, weakly complete, $\ldots$ \\ \hline
\kwstr $\k$ & ground field, sometimes assumed to be $\c$ \\ \hline
\kwstr $A(M; c)$ & skein space of the $n$-manifold $M$ with boundary condition $c$ (finite linear combinations
of $C$-string diagrams on $M$ with boundary $c$, modulo relations \\ \hline
\kwstr $Z(W)(s)$  & the path integral of the $n{+}1$-manifold $W$, evaluated on the string diagram $s$ in $\bd W$ \\ \hline
\kwstr $x$  & blah \\ \hline
\kwstr $x$  & blah \\ \hline
\end{tabular}
}

\bigskip

\noop{
\nn{Notes for intro:}
\begin{itemize}
\item Abstract: We define a universal state sum construction which specializes to most previously known
state sums (Turaev-Viro, Dijkgraaf-Witten, Crane-Yetter, Douglas-Reutter, Witten-Reshetikhin-Turaev, Brown-Arf).
\nn{and produces new examples, even in low dimensions?}
\item keep short(?)
\item comment on Vec and SVec; everything applies to SVec(?)
\item perhaps comment that this paper is sort of backwards, putting emphasis on the least interesting parts of the TQFT
\item notation summary table?  (for people inclined to skim)
\end{itemize}

\noindent
\nn{To do:}
\begin{itemize}
\item fermionic 2+1 TVish SS??  also super CY; also super DR?
\item SS for defects (at least codim-1)?
\item Turaev shadow state sum (at least mention) (?)
\item discuss lack of need for combings, orderings, etc. in TV
\item maybe remark on non-strict pivotal; expect non-strict case to work; 
need to be more careful about defining string diagrams in this case;
perhaps cf Reutter-W (in preparation) (at least for low dims)
\end{itemize}
}

\section{The state sum}
\label{ss-sect}

\subsection{The sum}
\label{ss-sum}

Let $C$ be a linear, $H$-pivotal, finite, weakly complete, 
semisimple $n$-category equipped with conjugation and a nondegenerate evaluation map, 
as defined in Subsection \ref{n-cat}.
Let $W$ be an $n{+}1$-dimensional $H$-manifold, and let $\cH$ be a cell decomposition of $W$.
From $\cH$ one can construct a handle decomposition of $W$
(by thickening the cells),
and we will use the same notation to refer to the cell decomposition 
(and its constituent cells) and the handle decomposition (and its constituent handles).
Let $\cL(\cH)$ denote the (finite) set of labelings of the $j$-cells/$j$-handles of $\cH$ by minimal
$(n{+}1{-}j)$-morphisms of $C$ (for $1\le j\le n+1$), as described in \ref{ss-labelings}.
Define
\[
	Z(W) = \sum_{\beta\in\cL(\cH)} \;\; \prod_{j=0}^{n+1}\;\;  \prod_{h\,\in\, \text{$j$-handles}} 
			\frac{\ev(\beta(\bd h))}{N(\beta(h))}.
\]
Then $Z(W)$ is independent of the choice of cell decomposition $\cH$, and in fact $Z(\cdot)$ comprises
the $n{+}1$-dimensional part (path integral) of a fully extended TQFT (see Section \ref{s-np1} and Appendix \ref{s-nep}).

\medskip

In the remainder of this subsection I'll give a brief explanation
(targeted at experts and the impatient) of the notation used in the above state sum.
Later subsections will give more details.

\medskip

This paper is to some extent agnostic as to what model of $n$-categories is used, so long as that model supports
the construction of string diagrams on $H$-manifolds.
All arguments of this paper are in terms of string diagrams, so if your favorite model of $n$-categories 
affords string diagrams,
then the arguments herein apply to that model.

$H$-pivotal (where $H$ is $SO(n)$, $O(n)$, $Spin(n)$, $Pin_\pm(n)$, etc.) means, roughly, that 
the morphisms of $C$ are equipped with an action of automorphisms of $n$-balls with an $H$ structure
(e.g. oriented balls, spin balls, etc.).
(See \cite{MWblob} and \ref{n-cat} below.)
$H$-pivotal $n$-categories have enough data satisfying enough coherence relations so that one 
can define string diagrams on $H$ manifolds.
When $n=2$ and $H = SO(2)$, $H$-pivotal is just the usual notion of (strict) pivotal tensor categories (and 2-categories).
One expects that $H$-pivotal categories correspond to (a strictified version of) $H$ homotopy fixed points in the sense of \cite{LurieCH}.

``Conjugation", as used above, means a (possibly anti-linear) isomorphism between 
the $n$-morphisms assigned to an $n$-ball B
and $n$-morphisms assigned to the orientation-reversed ball $\rev B$.
For examples arising from quantum groups, conjugation amounts to ``reversing arrows" or, alternatively,
changing labels from $\alpha$ to $\alpha^*$.
If $x$ is an $n$-morphism of shape $B$, and $\ol x$ is its conjugate (of shape $\rev B$), then $x$ and $\ol x$ can be
glued together to form a string diagram on $B \cup \rev B \cong S^n$, and this diagram can be evaluated using
the supplied evaluation map to give a scalar $\ev(x\cup \ol x)\in\k$.

The evaluation map (secretly, the path integral of the $n{+}1$-ball 
$Z(B^{n+1})$) is a map from finite linear combinations of string diagrams
on $S^n$ to the ground field $\k$.
``Nondegenerate" means the pairing defined in the previous paragraph is nondegenerate for all fixed boundary conditions
on $\bd B^n$.
(Boundary conditions can be thought of as the combined source and target of the $n$-morphism $x$.
Because of pivotality, there's not much point in distinguishing between source and target.)

For much of this paper, $\k$ can be any field.
But when we speak of simple objects and idempotents we will assume that $\k = \c$, unless specified otherwise.

Let $x$ be a $k$-morphism of $C$, with $k < n$.
Let $\id^{n-k-1}(x)$ denote the $(n{-}k{-}1)$-times iterated identity of $x$, an $(n{-}1)$-morphism of $C$.
The endomorphisms of $\id^{n-k-1}(x)$ form an algebra (commutative if $k < n-1$), 
which is semisimple because of our assumptions on $C$.
We define $x$ to be minimal if these endomorphisms are a simple algebra.
We define $C$ to be ``weakly complete" if every $k$-morphism is isomorphic to a sum of minimal $k$-morphisms (for all $k$).
If $C$ is not weakly complete, then it can be completed to a Morita equivalent $n$-category
which is weakly complete.

We define minimal $k$-morphisms $x$ and $x'$ to be equivalent if there exists a non-zero $k{+}1$-morphism
connecting $x$ to $x'$.
Note that this is a coarser equivalence relation than the usual notion of $n$-categorical equivalence.

The key property of minimal $k$-morphisms is the following:
given a string diagram $c$ on $\bd B^k$, the equivalence classes of minimal $k$-morphisms with source/target $c$ index an orthogonal
basis of $A(B^k\times S^{n-k}; c\times S^{n-k})$
(This assumes that $\k=\c$ and that we are enriched in vector spaces rather than super vector spaces.
More generally the basis is in indexed by pairs $(m, \alpha)$ where $m$ is a minimal $k$-morphism with source/target $c$
and $\alpha$ is a basis vector of the endomorphisms of $\id^{n-k-1}(m)$.)

The labelings $\cL(\cH)$ are constructed sequentially, starting with the top-dimensional cells.
The $n{+}1$-cells are labeled by (a set of representatives of the equivalence classes of) the minimal 0-morphisms of $C$.
Each $n$-cell is labeled by minimal 1-morphisms in $\mor^1(x\to y)$, where $x$ and $y$ are the labels previously assigned to
the two $n{+}1$-cells adjacent to the $n$-cell.
Each $(n{-}1)$-cell is labeled by a minimal 2-morphism of $C$ whose boundary is determined by the $n{+}1$-cells and $n$-cells
adjacent to the $(n{-}1)$-cell.
And so on.
At each stage, the labels previously chosen determine a $C$-string diagram on the linking $(n{-}k)$-sphere
of the $k$-cell, this string diagram determines a set of $(n{-}k{+}1)$-morphisms of $C$, and we choose labels
from a set of representatives of the equivalence classes of minimal morphisms in that set.

The intersection of the boundary of each handle $h$ with the underlying cell decomposition determines
an unlabeled cell complex in the $n$-sphere $\bd h$.
A labeling $\beta\in\cL(\cH)$ converts this unlabeled cell complex into a labeled string diagram, denoted $\beta(\bd h)$, 
in $S^n$.
In the state sum formula above, $\ev(\beta(\bd h))$ denotes the evaluation of this diagram.

All of the ingredients of the state sum discussed so far could be easily guessed from knowledge of the Turaev-Viro and 
Crane-Yetter state sums.
The normalization factor $N(\beta(h))$ is less obvious.
Let $x$ be a $k$-morphism of $C$.
Secretly, $N(x)$ is equal to the TQFT inner product of $x\times S^{n-k}$ with itself.
Officially, we define (inductively)
\[
	N(x) = \sum_{y} \frac{\tr_s(y)^2}{N(y)}
\]
where $y$ runs through minimal endomorphisms of $x$ and the ``sphere-trace" $\tr_s$ is defined as
\[
	\tr_s(y) = \ev(y \times S^{n-k-1} \cup (\bd y)\times B^{n-k}) .
\]
Despite the complicated-looking definition, $\tr_s(y)$ should be thought of as the simplest possible way of assigning
a number to a $k$-morphism using the evaluation map for diagrams on the $n$-sphere.
See \ref{ss-norm} for more details.
To get the induction started, we define $N(x)$ to be the square of the norm-square
of $x$ (using the inner product defined above)
when $x$ is an $n$-morphism.
In low codimension, one computes that 
\begin{itemize}
	\item for $k = n-1$, $N(x) = \dim(\End(x))$
	\item for $k = n-2$, $N(x)$ is the global dimension (sum of squares of simple objects) of the 
	tensor category of endomorphisms of $x$ (up to a scaling factor)
	\item for $k = n-3$, $N(x)$ is the sum of reciprocals of global dimensions of the constituent tensor categories
	of the 3-category $\End(x)$
\end{itemize}

\medskip

If $W$ has nonempty boundary and $x$ is a $C$-string diagram on $\bd W$ we can extend the state sum to
\[
	Z(W)(x) = \sum_{\beta\in\cL_x(\cH)} \;\; \prod_{j=0}^{n+1}\;\;  \prod_{h\,\in\, \text{$j$-handles}} 
			\frac{\ev(\beta_x(\bd h))}{N(\beta(h))}.
\]
Most of the ingredients are the same as in the closed case, except the string diagram that we evaluate on each
handle, $\beta_x(\bd h)$, depends on both the labeling $\beta$ and the given string diagram $x$.
If a handle $h$ does not intersect $\bd W$ then $\beta_x(\bd h)$ is the same as $\beta(\bd h)$ above,
but if $h$ does intersect $\bd W$ then the string diagram $\beta_x(\bd h)$ 
on $\bd h \cap \bd W$ is the restriction of $x$ to $\bd h \cap \bd W$.
The set of labelings $\beta\in\cL_x(\cH)$ of handles is constrained to be compatible with $x$ near $\bd W$.

\medskip

The remainder of Section \ref{ss-sect} gives a more detailed
description of the ingredients of the state sum.
Some readers might want to first read Section \ref{s-spcase}, which works out the state
sum for various small values of $n$.

\subsection{$n$-categories}
\label{n-cat}


This subsection lists the requirements for the input $n$-category.

As noted above, 
this paper attempts to be agnostic as to what model of $n$-categories is used.
If a model supports the construction of string diagrams on $H$-manifolds, then the constructions and proofs of this paper
should apply.

Categories (meaning $n$-categories) having the right sort of duality (in dimensions 0 through $n$) include:
\begin{itemize}
\item strict pivotal tensor categories (or more generally strict pivotal 2-categories)
\item ribbon categories (thought of as 3-categories)
\item $\pi_{\le n}(T)$ for any topological space $T$ (e.g. $T = BG$, for $G$ a finite group; 
or $T$ a space with $\pi_k(T)$ finite for all $k$)
\item disklike categories, as defined in \cite{MWblob}
\item string diagram categories, as defined in \cite{MWcompl}
\end{itemize}

We further assume that for each $k$, $0\le k\le n$, we have a ``conjugation" or ``orientation reversal" map which takes
a $k$-morphism $x$ of shape $X$ and produces a $k$-morphism $\trev x$ of shape $\trev X$.
(Here the ball $\trev X$ is best thought of as $X$ with a reversed normal bundle.)
The $k$-morphisms $x$ and $\trev x$ can be glued together to yield a string diagram on $S^k$.
The conjugation maps can be linear or anti-linear.

For categories satisfying the above conditions, we can construct a fully extended
$n{+}\epsilon$-dimensional TQFT, as explained in Appendix \ref{s-nep}, \cite{W2006} and \cite{MWblob}.
In particular, for each $n$-manifold $M$ and string diagram $c$ on $\bd M$ we have a pre-dual Hilbert space
$A(M;c)$.
This is defined to be string diagrams on $M$, restricting to $c$ on $\bd M$, modulo local relations.
For each $(n{-}1)$-manifold $Y$ and string diagram $c$ on $\bd Y$ we have a linear 1-category $A(Y;c)$.
The objects of $A(Y;c)$ are string diagrams on $Y$ (restricting to $c$ on $\bd Y$), and the morphisms
from $x$ to $y$ of $A(Y;c)$ are $A(Y\times I; \ol x \cup y)$.
See Appendix \ref{s-nep} for more details.

The conditions stated so far suffice to define the 0- through $n$-dimensional parts of an $n{+}1$-dimensional TQFT.
To get the $n{+}1$-dimensional part (Theorem \ref{pithm}), we make the following additional assumptions:
\begin{itemize}
\item The pre-dual Hilbert space $A(S^j\times B^{n-j}; c)$ is finite dimensional for all $j$ and for all
string diagrams $c$ on $\bd(S^j\times B^{n-j})$.
\item The linear 1-category $A(S^j\times B^{n-j-1}; c)$ is finite semisimple for all $j$ and for all
string diagrams $c$ on $\bd(S^j\times B^{n-j-1})$.
\end{itemize}

\noop{\nn{maybe remark that this implies for all manifolds (check first)}}

The final piece of input data for the state sum is a choice of evaluation map
\[
	\ev : A(S^n) \to \c .
\]
(Secretly, this is the path integral of $B^{n+1}$.)
We require that if an $H$-isomorphism $f: S^n\to S^n$ extends to $B^{n+1}$, then $\ev(f(x)) = \ev(x)$ for all string diagrams $x$.
For each string diagram $c$ on $S^{n-1}$, we get a pairing
\begin{eqnarray*}
	A(B^n; c)\ot A(B^n; c) & \to & \c \\
	x\ot y & \mapsto & \langle x, y\rangle \deq \ev(\ol x \cup y) .
\end{eqnarray*}
We require that the above induced pairings are nondegenerate for all $c$.
(We remark that if we further assume that the pairings are positive definite for all $c$, then the semisimplicity
requirement would be a consequence; see \cite{W2006}.)

The above assumptions suffice to extend the $n{+}\epsilon$-dimensional TQFT to a full $n{+}1$-dimensional TQFT; see \ref{pithm}.
To write the state sum in a convenient form, we further assume that the input $n$-category $C$ is ``weakly complete"
in the following sense.
Let $a$ be a $k$-morphism of $C$.
The algebra (1-category with one object) $\End(\id^{n-k-1}(a))$ is, by assumption, semisimple.
If it is simple, we say that $a$ is minimal.
If, for all $0\le k \le n-1$, all $k$-morphisms $a$ are isomorphic to a sum of minimal $k$-morphisms, we
say that $C$ is weakly complete.
If $C$ is not weakly complete (but satisfies the other conditions above, including in particular the semisimplicity
condition), then we can construct a new $n$-category $\ol C$, the ``weak completion" of $C$,
by adding a new $k$-morphism for each minimal idempotent in $\End(\id^{n-k-1}(a))$ (as above).
There is a Morita equivalence between $C$ and $\ol C$, so they lead to isomorphic TQFTs.
(Weak completion, as defined here, is a special case of the more complete form of completion discussed in \cite{MWcompl}.
The reader can find more details on completion there.
See also \cite{GJF2019}, which is similar in spirit.)

For example, a multi-fusion category $C$ is weakly complete if and only if the tensor unit is a simple object,
i.e.\ if and only if it is a fusion category.
The weak completion of a multi-fusion category is obtained by adding a new 0-morphism for each simple
summand of the tensor unit.

If we did not impose the weakly complete assumption, we could still write a state sum for the path integral, 
but it would be somewhat messier than the state sum in \ref{ss-sum}.

It bears repeating that we define minimal $k$-morphisms $a$ and $a'$ to be equivalent if there exists a non-zero $k{+}1$-morphism
connecting $a$ to $a'$.
Note that this is a coarser equivalence relation than the usual notion of $n$-categorical equivalence.

\noop{
\bigskip
\noindent \nn{to do:}
\begin{itemize}
\item remark that these are not necessarily weakest possible hypotheses
\item additive (Cauchy?) completeness
\end{itemize}
}

\subsection{Labelings}
\label{ss-labelings}

The basic idea of labelings is simple.

The labelings $\cL(\cH)$ are constructed sequentially, starting with the top-dimensional cells.
The $n{+}1$-cells are labeled by (a set of representatives of the equivalence classes of) the minimal 0-morphisms of $C$.
Each $n$-cell is labeled by minimal 1-morphisms in $\mor^1(x\to y)$, where $x$ and $y$ are the labels previously assigned to
the two $n{+}1$-cells adjacent to the $n$-cell.
Each $(n{-}1)$-cell is labeled by a minimal 2-morphism of $C$ whose boundary is determined by the $n{+}1$-cells and $n$-cells
adjacent to the $(n{-}1)$-cell.
And so on.
At each stage, the labels previously chosen determine a $C$-string diagram on the linking $(n{-}k)$-sphere
of the $k$-cell, this string diagram determines a set of $(n{-}k{+}1)$-morphisms of $C$, and we choose labels
from a set of representatives of the equivalence classes of minimal morphisms in that set.

But there are some subtleties.

To each $k$-cell $h$ we can associate a normal fiber $N_x$ (isomorphic to $B^{n+1-k}$), 
and this normal fiber has a cone-like cell decomposition.
The label assigned to $x$ should be thought of as a string diagram label for the central cone point
of this cone-like cell decomposition of $N_x$.
The other cells in the cell decomposition of $N_x$ are labeled according to the previously
chose labels of $m$-cells $y$ ($m > k$).
To do this we need to take into account an isomorphism between the normal fiber $N_y$ of $y$
and the corresponding normal fiber $P$ of a cell in the cell decomposition of $N_x$.
If we have chosen trivializations or standard models for these normal fibers,
those trivializations will not necessarily agree under the isomorphisms of normal fibers coming from the 
geometry of the cell decomposition.
For example
\begin{itemize}
\item In the $n=2$ Turaev-Viro case (see \ref{ss-tv}), a 1-cell $y$ is labeled by some $\alpha\in V_{abc}$.
Let $x$ and $x'$ be the two 0-cells adjacent to $y$, and let $P$ and $P'$ be the two normal fibers isomorphic to $N_y$
inside $N_x$ and $N_{x'}$.
If the orientation of $N_y$ agrees with that of $P$, it will disagree with the orientation of $P'$, and $P$
will be labeled by $\alpha$ while $P'$ is labeled by $\trev \alpha \in V_{c^*b^*a^*}$.
\item If $a=b=c$ in the previous example there is a further $\z/3$ rotational ambiguity in identifying $N_y$ with $P$,
and this must be taken into account.
\item In the $H$ = Spin case, the spin structure of $W$ will affect the normal fiber isomorphisms, and this is how the state
sum is sensitive to the spin structure (see \ref{ss-1d}).
Similar things are true for unoriented and pin manifolds.
\end{itemize}

In traditional approaches to the Turaev-Viro state sum, one sometimes chooses a global ordering 
of the vertices of the triangulation.
This choice of global ordering indirectly determines trivializations of normal fibers as above.

\noop{
\bigskip
\noindent \nn{to do:}
\begin{itemize}
\item role of or-rev thing ($\psi$?)
\item disklike versus traditional skeletonized; trivializations of normal fibers
\item maybe simple examples
\end{itemize}
}

\subsection{Norms}
\label{ss-norm}

First we define the ``sphere trace" $\tr_s(\cdot)$.
Let $x$ be a $k$-morphism of $C$.
As usual, we will identify $x$ with a cone-like string diagram on $B^k$.
To define $\tr_s(x)$,
we construct a string diagram on $S^n$ in the simplest way possible and then evaluate it.
Specifically, we define
\[
	\tr_s(x) \deq \ev(x \times S^{n-k} \cup (\bd x)\times B^{n-k+1}) .
\]
In words, from $x$ we can form $x \times S^{n-k}$, a string diagram on $B^k\times S^{n-k}$.
From $\bd x$, a string diagram on $S^{k-1}$, we can form $(\bd x)\times B^{n-k+1}$, a
string diagram on $S^{k-1}\times B^{n-k+1}$.
These two string diagrams can be glued together to construct a string diagram on $S^n$.

For small $n$, we have:
\begin{itemize}
\item $n=1, k=1$.  This is the evaluation (ordinary trace) of $\ol x x$, $\ev(\ol x x)$.
\item $n=1, k=0$.  This is the evaluation of $S^1$, labeled by $x$; the ordinary trace of $\id_x$.
\item $n=2, k=2$.  This is $\ev(\ol x \cup x)$, the evaluation of a ``doubled" version of $x$ (see Figure \ref{xxxx}).
\item $n=2, k=1$.  In a fusion category with standard evaluation, this is the quantum dimension $d_x$.
More generally, for a 2-category, and for $x: p\to q$, this is the right quantum dimension of $x$ times the evaluation of $S^2$ labeled by $p$, or the left quantum dimension of $x$ times the evaluation of $S^2$ labeled by $q$.
\item $n=2, k=0$.  The is the evaluation of the ``empty" string diagram on $S^2$, colored by $x$.
\item $n=3, k=3$.  If the boundary of $x$ is a graph $\Gamma$ in $S^2$ 
(for example, one popular choice of $\Gamma$ is a tetrahedron),
then $\tr_s(X)$ is the evaluation of the double cone on $\Gamma$ (a 2-complex embedded in $S^3$),
with appropriate labeling.
The labels of the two cone points are $x$ and $\ol x$.
\item $n=3, k=2$.  If the boundary of $x$ is a string diagram on $S^1$ with $j$ points labeled by
1-morphisms $y_1,\ldots, y_j$, then $tr_s(x)$ is the evaluation of a 2-complex (in $S^3$) built out of one circle and $j$ disks.
The disk labels are $y_1,\ldots, y_j$.
\item $n=3, k=2$.  If $x: p\to q$, then $tr_s(x)$ is the evaluation of a 2-sphere in $S^3$.
The 2-sphere is labeled by $x$ and the two adjacent 3-cells are labeled by $p$ and $q$.
\item $n=3, k=0$.  The is the evaluation of the ``empty" string diagram on $S^3$, colored by $x$.
\end{itemize}

\medskip

As stated above, the norm $N(x)$ of a $k$-morphism $x$ of $C$ is defined to be
\[
	N(x) \deq \sum_{y} \frac{\tr_s(y)^2}{N(y)}
\]
where $y$ runs through a (finite) set of of representatives of equivalence classes of minimal
endomorphisms of $x$.

Because of the $N(y)$ in the denominator of the right hand side, this is an inductive definition.  To get the induction
started we define $N(x) = \tr_s(x)^2$ 
when $x$ is an $n$-morphism.

As we will see in \ref{ss-pi2ss}, the above definition of $N(x)$ is simply the computation of the inner product
$\langle x\times S^{n-k},  x\times S^{n-k} \rangle$ in the TQFT built out of $C$.

For small $n$ it is straightforward to compute the norm:
\begin{itemize}
\item $n=1, k=0$.  In this case $N(x)$ is the dimension of the endomorphism algebra of $x$.
If $x$ is minimal and the enriching category is Vec, then $N(x)=1$.
If the enriching category is super vector spaces, then $N(x)=1$ for ordinary simple objects and $N(x)=2$ for Majorana
simple objects.
\item $n=2, k=1$.  This is again the dimension of the endomorphism algebra of $x$.
\item $n=2, k=0$.  For $C$ a fusion category, $x$ the standard 0-morphism, and the standard evaluation, we have 
$N(x) = \gdim(C) = \sum_y d^2_y$.
(The sum is over simple objects $y$.)
If $C$ is a super fusion category, we have
\[
	N(x) = \gdim(C) = \sum_y \frac{d^2_y}{N(y)} .
\]
If $C$ is a general 2-category, we have $N(x) = \gdim(\End(x))\cdot\ev(\eset_x)^2$, where
$\End(x)$ denotes the endomorphism tensor category of $x$ and $\eset_x$ is the empty string diagram on $S^2$ with label $x$.
\item $n=3, k=2$.  This is again the dimension of the endomorphism algebra of $x$.
\item $n=3, k=1$.  This is $\gdim(\End(x))\cdot\ev(S^2_x)^2$, where $\End(x)$ is the endomorphism tensor category of $x$
and $S^2_x$ is the string diagram in $S^3$ consiting of an embedded 2-sphere labeled by $x$ (and two 3-cells labeled by the
domain and range of $x$). 
\item $n=3, k=0$.  In this case we have
\[
	N(x) = \sum_y \frac{1}{\gdim(y)} 
\]
where the sum is over minimal endomorphisms of $x$.
\end{itemize}

Note that when $n-k$ is odd, $N(x)$ does not change when one scales the evaluation function, and when $n-k$ is even
rescaling the evaluation function by $\lambda$ multiplies $N(x)$ by $\lambda^2$.
This is because the Euler characteristic of $S^{n-k}$ is zero [two] when $n-k$ is odd [even].

\section{Special cases}
\label{s-spcase}

The discussion of most of the special cases below follows the same pattern:
\begin{itemize}
\item Describe the labelings $\beta\in\cL(\cH)$ of the handles/cells.
\item Compute the handle-boundary evaluations $\ev(\beta(\bd h))$.
\item Compute the norms $N(\beta(h))$ of the handle labels.
\item Plug the above information into the universal state sum formula and recognize the result as a familiar, 
previously known state sum.
(In some cases we will have a new state sum.)
\end{itemize}

\subsection{Turaev-Viro}
\label{ss-tv}

Let $n=2$ and $H=SO(2)$.
We will initially assume that our pivotal 2-category is a pivotal fusion category
(only one 0-morphism, which is assumed to be minimal).

Let's first assume that our cell decomposition is generic (i.e.\ dual to a triangulation); three 2-cells adjacent to each
1-cell, four 1-cells and six 2-cells adjacent to each 0-cell (in a tetrahedral pattern).

The labelings $\beta\in\cL(\cH)$ in this case are as follows.
\begin{itemize}
\item 3-cells are labeled by minimal 0-morphisms of $C$, of which there is only one
(denoted $*$), because $C$ is a fusion category.
\item 2-cells are labeled by minimal 1-endomorphisms of the unique minimal 0-morphism $*$, which
are just the simple objects of the fusion category.
\item 1-cells are labeled by an orthogonal basis of $\mor(\tunit \to a\ot b\ot c)$, where $a$, $b$ and $c$ are
the labels (or duals thereof) assigned to the 2-cells adjacent to the 1-cell.
(Whether or not we use dual labels depends on a choice of orientation of the 2-cell relative to the 1-cell.)
\noop{\nn{discuss this further below or elsewhere?}}
\end{itemize}

\kfigb{tv-fig}{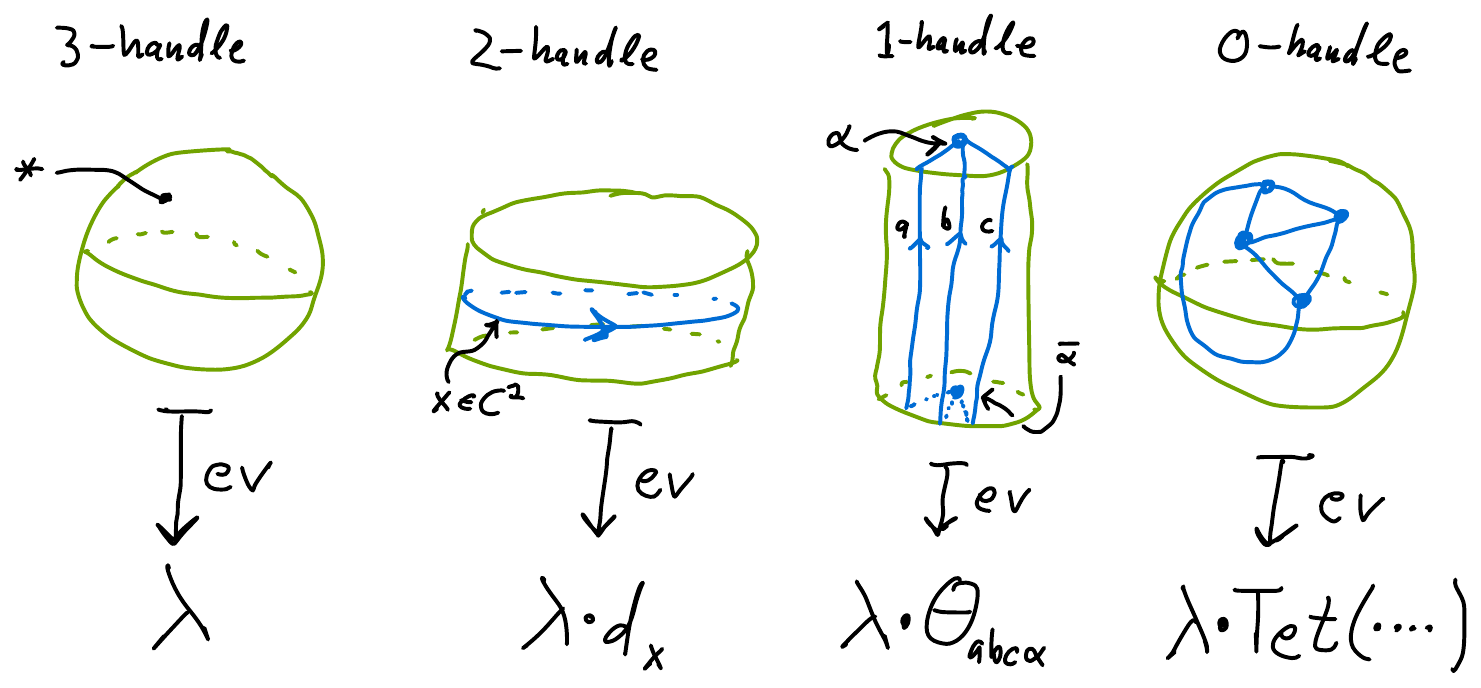}{Handle evaluations for the Turaev-Viro state sum}{6in}

The handle-boundary evaluations $\ev(\beta(\bd h))$ are as follows. 
(See Figure \ref{tv-fig}.)
\begin{itemize}
\item The boundary of a 3-handles is a 2-sphere labeled by $*$ -- in other words, the empty picture.
This evaluates to some scalar $\lambda\in\c$.
\item The boundary of a 2-handle $h$ is a 2-sphere with a single loop labeled by the simple object $\beta(h)$.
The evaluation is the quantum dimension (loop value) $d_{\beta(h)}$ times $\lambda$.
\item The boundary of a 1-handle $h$ is a ``theta" graph with three edges (labeled with 2-handle 
labels $a$, $b$ and $c$) and two vertices labeled by $\alpha = \beta(h)$ and $\ol{\alpha}$.
This is the usual ``theta" symbol of the fusion category, $\lambda \Theta_{abc\alpha}$.
\item The boundary of a 0-handle is a tetrahedral graph with labels coming from the six adjacent 
2-handles and four adjacent 1-handles.
The evaluation, up to normalization by theta symbols, is equal to a $6j$-symbol symbol of the fusion category, times $\lambda$.
We will denote it as $\lambda\ev(\text{Tet})$.
\end{itemize}

The norms of morphisms are as follows.
\begin{itemize}
\item The norm of a 1-handle label $\alpha$ is $(\ev(\alpha \cup \ol\alpha)^2) = (\lambda\Theta_{abc\alpha})^2$.
\item The norm of a 2-handle label (simple object) is 1.
\item The norm of a 3-handle label is the $\lambda^2$-scaled global dimension $\lambda^2\cdot\gdim(C)$.
\end{itemize}

Putting this all together, we have, for a closed 3-manifold $W$, 
\begin{eqnarray*}
Z(W) &= & 
			\sum_{\beta\in\cL(\cH)} 
			\prod_{h_3} \frac{\ev(\beta(\bd h_3))}{N(\beta(h_3))}
			\prod_{h_2} \frac{\ev(\beta(\bd h_2))}{N(\beta(h_2))}
			\prod_{h_1} \frac{\ev(\beta(\bd h_1))}{N(\beta(h_1))}
			\prod_{h_0} \ev(\beta(\bd h_0))
			\\
&= &
			\sum_{\beta\in\cL(\cH)} 
			\prod_{h_3} \frac{\lambda}{\lambda^2\cdot\gdim(C)}
			\prod_{h_2} \frac{\lambda d_{\beta(h_2)}}{1}
			\prod_{h_1} \frac{\lambda \Theta_{abc\alpha}}{\lambda^2 \Theta_{abc\alpha}^2}
			\prod_{h_0} \lambda\ev(\text{Tet})
			\\
&=&			\lambda^{\chi(W)} \sum_{\beta\in\cL(\cH)} 
			\prod_{h_3} \frac{1}{\gdim(C)}
			\prod_{h_2} d_{\beta(h_2)}
			\prod_{h_1} \frac{1}{\Theta_{abc\alpha}}
			\prod_{h_0} \ev(\text{Tet})
\end{eqnarray*}
(The products are indexed by all $i$-handles $h_i$, $0\le i \le 3$.)
For closed 3-manifolds, the Euler characteristic $\chi(M)$ is zero, so there is no dependence on $\lambda$.
This is very close to the usual Turaev-Viro-Barrett-Westbury state sum \cite{TV92,BW99}
(for the triangulation dual to our generic cell decomposition).
One minor difference is that we do not require any ordering of the vertices of the dual triangulation.
Instead, the $\ev(\text{Tet})$ factors assigned to 0-handles will use potentially non-standard versions
of labeled tetrahedral graphs, depending on the choices of standardization made on adjacent 1- and 2-handles.

The state sum works equally well for general, non-generic cell decompositions.
In the general case, the factors associated to 3- and 2-handles are unchanged (because the linking spheres
of these low-codimension cells are the same as in the generic case).
For 1-handles, we replace the evaluation of a theta graph with the evaluation of mutant theta graph
(denoted $\Theta^m$), with two vertices
and an edge for each adjacent 2-handle.
For 0-handles, the tetrahedron is replaced by the graph on the linking 2-sphere (determined by adjacent 1- and 2-cells),
which could be arbitrarily complicated.
The labelings of a 1-handle are taken from an orthogonal basis of
$\mor(\tunit \to a_1\ot a_2\ot\cdots\ot a_k)$, where $k$ is the number of 2-handles adjacent to the 1-handle.
In summary, the state sum for an arbitrary cell decomposition has the form
\[
	Z(W) = \lambda^{\chi(M)} \sum_{\beta\in\cL(\cH)} 
			\prod_{h_3} \frac{1}{\gdim(C)}
			\prod_{h_2} d_{\beta(h_2)}
			\prod_{h_1} \frac{1}{\ev(\Theta^m(\bd h_1)}
			\prod_{h_0} \ev(\text{Link}(h_0))
\]

\noop{If $W$ has boundary, we can fix a $C$-labeled graph on the boundary
\nn{ ...  refer to general discussion above}

\nn{to do: more general input 2-cat; planar alg case??}
}

\subsection{Crane-Yetter}
\label{ss-cy}

Let $n=3$ and $H=SO(3)$, and assume that the 0- and 1-morphisms of $C$ are trivial.
Then $C$ is a premodular category.

Let's first assume that the cell decomposition is generic (dual to a triangulation);
three 3-cells adjacent to each 2-cell, four 2-cells and six 3-cells adjacent to each 1-cell (in a tetrahedral pattern),
five 1-cells, ten 2-cells, and ten 3-cells adjacent to each 0-cell in a 4-simplex pattern.

The labelings $\beta\in\cL(\cH)$ in this case are as follows. \noop{\nn{should do figures}}
\begin{itemize}
\item 4-cells are labeled by minimal 0-morphisms of $C$, of which there is only one,
denoted $*_0$.
\item 3-cells are labeled by minimal endomorphisms of the unique minimal 0-morphism $*_0$.
There is only one possibility, denoted $*_1$.
\item 2-cells are labeled by minimal endomorphisms of $*_1$.
There are just the simple objects of the premodular category $C$.
\item 1-cells are labeled by an orthogonal basis of 
$\mor(\tunit \to a_1\ot a_2\ot a_3\ot a_4)$, where the $a_i$ are
the labels (or duals thereof) assigned to the 2-cells adjacent to the 1-cell.
(Whether or not we use dual labels depends on a choice of orientation of the 2-cell relative to the 1-cell.)
We can think of the $a_i$ as labeling the four corners of a tetrahedron in the linking 2-sphere of the 1-cell.
If one wanted more similarity
to the original Crane-Yetter state sum, 
one could resolve the 4-valent vertex into two 3-valent vertices ($V_{a_1a_2a_3a_4} \cong \bigoplus_x V_{a_1a_2x}\ot V_{x^*a_3a_4}$)
in order to write this basis in terms of more familiar data for the premodular category $C$, 
but we do not choose to do so.
\end{itemize}

The handle-boundary evaluations $\ev(\beta(\bd h))$ are as follows.
(See Figure \ref{cy-fig}.)

\kfigb{cy-fig}{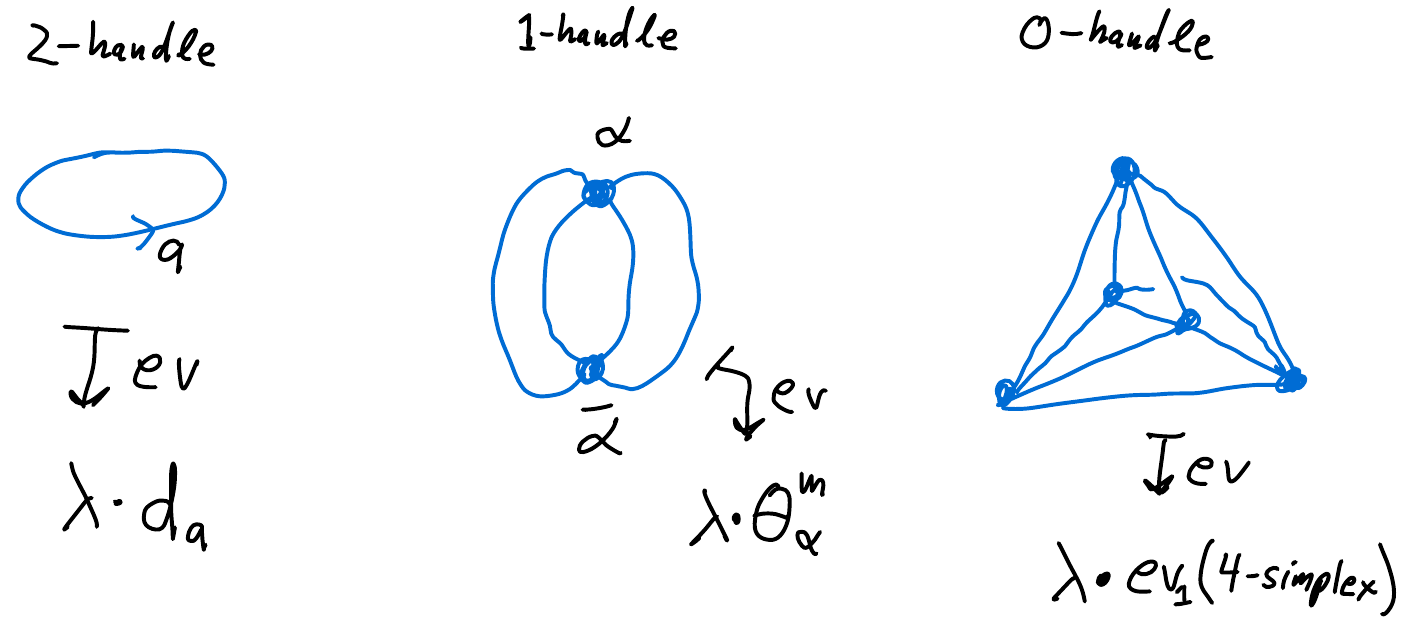}{Handle evaluations for the Crane-Yetter state sum}{5.5in}

\begin{itemize}
\item The boundary of a 4-handle is a 3-sphere labeled by $*_0$ -- in other words, the empty picture.
This empty string diagram evaluates to some scalar $\ev(\eset) = \lambda\in\c$.
\item The boundary of a 3-handle is again decorated by an empty string diagram (a 2-sphere in $S^3$ labeled by $*_1$, 
with two adjacent 3-balls labeled by $*_0$), 
and so evaluates to $\lambda$.
\item The boundary of a 2-handle $h$ is a 3-sphere with a single loop labeled by the simple object $\beta(h)$.
The evaluation is the quantum dimension (loop value) $d_{\beta(h)}$ times $\lambda$.
\item The boundary of a 1-handle $h$ is a ``mutant theta" graph with four edges (labeled with 2-handle 
labels $a_1,\ldots,a_4$) and two vertices labeled by $\alpha = \beta(h)$ and $\ol{\alpha}$.
(If we resolve the two 4-valent vertices, this mutant theta graph can be evaluated in terms of tetrahedral graphs.)
\item The boundary of a 0-handle is the 1-skeleton of a 4-simplex embedded in $S^3$.
There are ten simple object labels on the edges and five 4-valent vertex labels on the the vertices.
\end{itemize}

The norms of morphisms are as follows.
\begin{itemize}
\item The norm of a 1-handle label $\alpha$ is $\ev(\alpha \cup \ol\alpha)^2 = (\lambda\Theta^m_\alpha)^2$,
where $\Theta^m_\alpha$ denotes a ``mutant" theta symbol.
\item The norm of a 2-handle label (simple object) is 1
(assuming we are enriched in Vec and not SVec).
\item The norm of a 3-handle label is the $\lambda^2$-scaled global dimension $\lambda^2\gdim(C)$.
\item The norm of a 4-handle label is the reciprocal of the global dimension $1/\gdim(C)$.
\end{itemize}

Putting this all together, we have, for a closed 4-manifold $W$, 
\begin{eqnarray*}
Z(W) &= & 
			\sum_{\beta\in\cL(\cH)} 
			\prod_{h_4} \frac{\ev(\beta(\bd h_4))}{N(\beta(h_4))}
			\prod_{h_3} \frac{\ev(\beta(\bd h_3))}{N(\beta(h_3))}
			\prod_{h_2} \frac{\ev(\beta(\bd h_2))}{N(\beta(h_2))}
			\prod_{h_1} \frac{\ev(\beta(\bd h_1))}{N(\beta(h_1))}
			\prod_{h_0} \ev(\beta(\bd h_0))
			\\
&= &
			\sum_{\beta\in\cL(\cH)} 
			\prod_{h_4} \frac{\lambda}{1/\gdim(C)}
			\prod_{h_3} \frac{\lambda}{\lambda^2 \gdim(C)}
			\prod_{h_2} \frac{\lambda d_{\beta(h_2)}}{1}
			\prod_{h_1} \frac{\lambda \Theta^m_{\alpha}}{(\lambda \Theta^m_{\alpha})^2}
			\prod_{h_0} \lambda\ev(\text{4-simplex})
			\\
&=&			\lambda^{\chi(W)} \sum_{\beta\in\cL(\cH)} 
			\prod_{h_4} \gdim(C)
			\prod_{h_3} \frac{1}{\gdim(C)}
			\prod_{h_2} d_{\beta(h_2)}
			\prod_{h_1} \frac{1}{\Theta^m_{\alpha}}
			\prod_{h_0} \ev(\text{4-simplex})
\end{eqnarray*}

This is equivalent to the state sum in \cite{CYK97} if we resolve 4-valent vertices into pairs of 3-valent vertices
and set $\lambda=1$.

For a general (not dual to a triangulation) cell decomposition, we have the following modifications:
\begin{itemize}
\item The link of a 1-cell can be a general cell decomposition of $S^2$ (rather than a tetrahedron).
So the resulting mutant theta graphs will have as many edges as there are vertices in this cell decomposition.
\item The link of a 0-cell is a general graph in $S^3$ (rather than a 4-simplex.
\end{itemize}
So for general cell decompositions we have
\[
Z(W) = 
			\lambda^{\chi(W)} \sum_{\beta\in\cL(\cH)} 
			\prod_{h_4} \gdim(C)
			\prod_{h_3} \frac{1}{\gdim(C)}
			\prod_{h_2} d_{\beta(h_2)}
			\prod_{h_1} \frac{1}{\Theta^m_{\alpha}}
			\prod_{h_0} \ev(\text{Link}(h_0)) .
\]

\medskip

As explained in \cite{W2006}, if $C$ is modular and 
we choose $\lambda = \pm (\gdim(C))^{-1/2}$, then the resulting $3{+}1$-dimensional TQFT
is bordism invariant, and the Witten-Reshetikhin-Turaev theory for $C$ can be recovered by considering manifolds with 
boundary.
We have, for $X$ a $k$-manifold and $k=0\ldots3$,
\[
	Z_{WRT}(X) = Z_{3{+}1}(\bd^{-1}(X))(\eset),
\]
where $\eset$ denotes the empty string diagram on $X$.
This is defined only when $\bd^{-1}(X)$ exists; $X$ can be any 1-, 2- or 3-manifold, but if $X$ is an oriented 0-manifold, it must
have the same number of positive and negative points.
Manifolds $X$ must equipped with extra structure to determine $\bd^{-1}(X)$ up to (iterated) bordism.

\subsection{WRT surgery formula}
\label{ss-wrtsf}

Let's again consider the case $n=3$, $H=SO(3)$, and the 0- and 1-morphisms of $C$ are trivial
(premodular category).

Let $W$ be a 4-manifold (with boundary) built out of a single 0-handle and some 2-handles, attached
to the 0-handle along a framed link $L\subset S^3$.
The boundary of $W$ is the 3-manifold obtained from Dehn Surgery on $L$.

We will choose the empty string diagram as a boundary condition on $W$.

The set of labelings $\beta(\cH)$ are assignments of a simple object $\beta(h)$ to each 2-handle $h$.

When $h$ is a 2-handle, we have
\[
	\ev(\beta(\bd h)) = \lambda d_{\beta(h)} ,
\]
where $\lambda$ is, as usual, the evaluation of the empty string diagram.

When $h$ is the 0-handle, we have
\[
	\ev(\beta(\bd h)) = \ev(\beta(L)) = \lambda J(\beta(L)) ,
\]
where $\beta(L)$ denotes the string diagram obtained by labeling each component of $L$ according to $\beta$,
and $J$ is the generalized Jones polynomial (normalized so that $J$ of the empty link is 1).

As before, then norms of the 2-handle labels are all 1.

Putting this all together, we have
\[
	Z(W)(\eset) = \sum_{\beta\in\cL(\cH)} 
			\prod_{h_2} \lambda d_{\beta(h_2)}
			\prod_{h_0} \lambda J(\beta(L)) .
\]
If $C$ is a modular category and we choose $\lambda^2 = \gdim(C)^{-1}$, then this is the usual
Witten-Reshetikhin-Turaev surgery formula for $Z_{WRT}(\bd W)$
\cite{Wit_jones,RT1991}.

\subsection{Douglas-Reutter}
\label{ss-dr}

Let n = 3 and H = SO(3).
Assume that the 0-morphisms of $C$ are trivial.
This is the monoidal 2-category case considered in \cite{DRf2c}.

As before, we'll first consider the case of a generic cell decomposition, then a general cell decomposition.

For a generic cell decomposition (see \ref{ss-cy}), the labelings $\beta\in\cL(\cH)$ are as follows. 
\noop{\nn{should do figures}}
\begin{itemize}
\item 4-cells are labeled by minimal 0-morphisms of $C$, of which there is only one,
denoted $*$.
\item 3-cells are labeled by minimal endomorphisms of the unique minimal 0-morphism $*$.
$\End(*)$ is a 2-category, and because of our weak completeness assumption, it is a sum of indecomposable 2-categories.
There is one equivalence class of minimal endomorphism of $*$ per summand.
\item Each 2-cell is adjacent to three 3-cells, with labels $a_1$, $a_2$ and $a_3$.
There is a 1-category of morphisms from $\id_*$ to $a_1 \ot a_2 \ot a_3$, which we will denote 
$\mor(\id_* \to a_1 \ot a_2 \ot a_3)$.
If the 3-cells and 2-cells are oriented compatibly, then 
the 2-cell is labeled by simple objects of $\mor(\id_* \to a_1 \ot a_2 \ot a_3)$.
In general, we take simple objects of
$\mor(\id_* \to a_1 \ot a_2^* \ot a_3)$, $\mor(\id_* \to a_1^* \ot a_2^* \ot a_3)$, etc., 
depending on the relative orientations of 3- and 2-cells.
\item 1-cells are labeled by an orthogonal basis of 
$\mor(\beta(\text{Tet}))$, where $\beta(\text{Tet})$ denotes a labeled 
tetrahedral graph in the linking 2-sphere of the 1-cell,
and $\mor(\beta(\text{Tet}))$ denotes the vector space of 3-morphisms corresponding to this string diagram on the 2-sphere.
The tetrahedral graph has four vertices corresponding to the four adjacent 2-cells, and six
edges corresponding to the six adjacent 3-cells.
The labels (or duals thereof, depending on relative orientations) 
assigned to the adjacent 2- and 3-cells determine the labelings of the 2-complex.
\end{itemize}

The handle-boundary evaluations $\ev(\beta(\bd h))$ are as follows. 
\noop{\nn{figures ??}}
\begin{itemize}
\item The boundary of a 4-handle is a 3-sphere labeled by $*$ -- in other words, the empty picture.
This empty string diagram evaluates to some scalar $\ev(\eset) = \lambda\in\c$.
\item The boundary of a 3-handle is a 2-sphere (in $S^3$) labeled by the minimal morphism $\beta(h)$.
We will denote the evaluation by $\lambda\ev(S^2_{\beta(h)})$.
\item The boundary of a 2-handle $h$ is a ``spun-theta" 2-complex, built out of a circle and three 2-cells.
(This is a generalization of loop evaluation that appears in the Crane-Yetter state sum.)
We will denote the evaluation by $\lambda\ev(\text{spun-}\Theta)$.
\item The boundary of a 1-handle $h$ is 2-complex that can be thought of as a double cone on a tetrahedron.
It has two vertices (the two cone points), four 1-cells, and six 2-cells.
(If one ignores the 2-cell labels, then this reduces to the four-barred mutant theta graph of the Crane-Yetter invariant.)
We will denote the evaluation by $\lambda\ev(\text{DCTet})$.
\item The boundary of a 0-handle is the 2-skeleton of a 4-simplex embedded in $S^3$.
There are 
five vertices (labeled by 1-cells labels described above),
ten edges (labeled by 2-cell labels),
and ten 2-cell faces (labeled by 3-cell labels).
We will denote the evaluation by $\lambda\ev(\text{4-simplex})$.
\end{itemize}

The norms of morphisms are as follows.
\begin{itemize}
\item The norm of a 1-handle label $\alpha$ is $\ev(\alpha \cup \ol\alpha)^2 = (\lambda\ev(\text{DCTet}))^2$,
where the DCTet 2-complex has its two cone points labeled by $\alpha$ and $\ol \alpha$.
\item The norm of a 2-handle label (simple object of $\mor(\id_* \to a_1 \ot a_2 \ot a_3)$) is 1
(assuming we are enriched in Vec and not sVec).
\item The norm of a 3-handle label $m$ (a minimal 1-morphism of $C$) is the $\lambda^2$-scaled global dimension 
of the tensor category $\End(m)$, time the evaluation of a 2-sphere labeled by $m$.
We will denote this by $\lambda^2\ev(S^2_m)^2\gdim(m)$.
\item The norm of a 4-handle label is 
\[
	\gdim_3(C) \deq \sum_m \frac{1}{\gdim(m)} ,
\]
where the sum is over equivalence classes of minimal 1-morphisms $m$.
\end{itemize}

Putting this all together, we have, for a closed 4-manifold $W$, 
\begin{eqnarray*}
Z(W) &= & 
			\sum_{\beta\in\cL(\cH)} 
			\prod_{h_4} \frac{\ev(\beta(\bd h_4))}{N(\beta(h_4))}
			\prod_{h_3} \frac{\ev(\beta(\bd h_3))}{N(\beta(h_3))}
			\prod_{h_2} \frac{\ev(\beta(\bd h_2))}{N(\beta(h_2))}
			\prod_{h_1} \frac{\ev(\beta(\bd h_1))}{N(\beta(h_1))}
			\prod_{h_0} \ev(\beta(\bd h_0))
			\\
&= &
			\sum_{\beta\in\cL(\cH)} 
			\prod_{h_4} \frac{\lambda}{\gdim_3(C)}
			\prod_{h_3} \frac{\lambda\ev(S^2_{\beta(h_3)})}{\lambda^2\ev(S^2_{\beta(h_3)})^2\gdim(\beta(h_3))}
			\prod_{h_2} \frac{\lambda\ev(\text{spun-}\Theta)}{1}
	\\ &  & \qquad\qquad\qquad\qquad		
			\prod_{h_1} \frac{\lambda\ev(\text{DCTet})}{(\lambda\ev(\text{DCTet}))^2}
			\prod_{h_0} \lambda\ev(\text{4-simplex})
			\\
&=&			\lambda^{\chi(W)} \sum_{\beta\in\cL(\cH)} 
			\prod_{h_4} \frac{1}{\gdim_3(C)}
			\prod_{h_3} \frac{1}{\ev(S^2_{\beta(h_3)}) \gdim(\beta(h_3))}
			\prod_{h_2} \ev(\text{spun-}\Theta)
	\\ &  & \qquad\qquad\qquad\qquad		
			\prod_{h_1} \frac{1}{\ev(\text{DCTet})}
			\prod_{h_0} \ev(\text{4-simplex})
\end{eqnarray*}
This is essentially the Douglas-Reutter state sum \cite{DRf2c}.
The main difference is that Douglas and Reutter use a finer equivalence relation on the minimal
1-morphisms $\beta(h_3)$ and introduce an additional normalization factor to compensate for the resulting redundancy in the sum.

For a general (not dual to a triangulation) cell decomposition, we have the following modifications:
\begin{itemize}
\item The $\ev(\text{spun-}\Theta)$ factor is replaced by the evaluation of a ``spun mutant theta"
2-complex, consisting of a circle and $k$ disks, one disk for each 3-handle adjacent to the 2-handle.
\item The DCTet 2-complex is replaced by the double cone of a general cell decomposition of the linking 2-sphere
of the 1-cell.
\item The $\ev(\text{4-simplex})$ factor is replaced by the evaluation of the 2-skeleton of a general cell
decomposition of $S^3$.
\end{itemize}

\subsection{$n=1$ cases (Euler characteristic, Brown-Arf, ...)}
\label{ss-1d}

\noop{
This subsection covers $1{+}1$-dimensional state sums for oriented, unoriented, Spin, 
$\mbox{Pin}_+$ and $\mbox{Pin}_-$ manifolds.
\nn{change if I don't do all those cases}
\nn{maybe summarize: Euler, Arf, Brown-Arf, others, ... }
}

\subsubsection{Oriented}

Let $H = SO(1)$, which is the trivial group.
Then $H$-pivotal 1-categories are just plain (linear, semisimple) 1-categories.
In this case ``weakly complete" is equivalent to being idempotent complete, and minimal 0-morphisms are simple objects.
Let $\{p_i\}_{i\in S}$ be a set of representatives of the equivalences classes of minimal objects.

\begin{itemize}
\item 2-cells are labeled $\{p_i\}$.
\item 1-cells are labeled as follows.
Let $p_i$ and $p_j$ be the labels of the two 2-cells adjacent to the 1-cell.
(If the  1-cell is part of the boundary $\bd W$, 
then one of these two labels will instead come from the specified boundary condition on $\bd W$.)
The 1-cell is labeled by an orthogonal basis of $\mor(p_i\to p_j)$.
If $i\ne j$ then this set is empty and this labeling of 2-cells does not contribute to the state sum.
If $i=j$ then this is a 1-dimensional vector space and we can, for convenience, choose $\id: p_i\to p_i$
as our basis vector.
\end{itemize}

It follows that the only labelings which contribute to the state sum are those in which
all 2-cells in the same connected component of $W$ are labeled with the same simple object $p_i$.
For simplicity we will now assume that $W$ is connected.

For each $i\in S$, let $a_i$ be the evaluation of the ``empty" string diagram on $S^1$ where all of $S^1$ 
is labeled by $p_i$.
For each fixed labeling and each 0-, 1- or 2-handle, the string diagram for the handle evaluation consists of regions
labeled by $p_i$ and (for 0- and 1-handles) points labeled by $\id:p_i\to p_i$.
It follows that each handle evaluation is equal to $a_i$ (for the $i$ determined by the labeling).

The norms of morphisms are as follows.
\begin{itemize}
\item The norm of a 1-cell label $\id:p_i\to p_i$ is $a_i^2$.
\item The norm of a 2-cell label $p_i$ is $\dim(\End(p_i)) = 1$.
\end{itemize}

Putting it all together, we have (for closed $W$)
\begin{eqnarray*}
Z(W) &=& 
			\sum_{\beta\in\cL(\cH)} 
			\prod_{h_2} \frac{\ev(\beta(\bd h_2))}{N(\beta(h_2))}
			\prod_{h_1} \frac{\ev(\beta(\bd h_1))}{N(\beta(h_1))}
			\prod_{h_0} \ev(\beta(\bd h_0))
			\\
&=&			\sum_{i\in S} 
			\prod_{h_2} \frac{a_i}{1}
			\prod_{h_1} \frac{a_i}{a_i^2}
			\prod_{h_0} a_i
			\\
&=&			\sum_{i\in S} a_i^{\chi(W)}
\end{eqnarray*}
where $\chi(W)$ is the Euler characteristic of  $W$.
(Recall that we are assuming that $W$ is connected.)

\subsubsection{Unoriented}

Now consider $H = O(1) \cong \z/2$.
An $H$-pivotal 1-category $C$ comes equipped with a linear anti-automorphism $r$, corresponding to the 
orientation-reversing map of $B^1$ to itself.

\noop{\nn{$r(e) = e'$ implies $a(e) = a(e')$}}

$r$ permutes the equivalence classes of minimal idempotents.
$C$ can be decomposed into a sum of $H$-pivotal 1-categories such that for each summand
there is a single $r$-orbit.

If this orbit has size 1 (trivial $r$ action), 
then things are very similar to the oriented case and we have
\[
	Z(W) = a^{\chi(W)},
\]
where $a$ is the evaluation of ``empty" diagram on $S^1$, as in the previous subsection.

\noop{\nn{add figure/example for following paragraph}}

If the orbit has size 2, then we have a minimal idempotent $e$ such that the minimal
idempotent $r(e)$ is orthogonal to $e$.
Note that the invariance property on the evaluation map $\ev: A(S^1) \to \k$ implies that the
evaluations of the closure of $e$ and $r(e)$ are both equal to the same $a\in \k$.
The handle-boundary evaluations for 0-handles will have a mixture of points labeled by
$e$ and $r(e)$.
(Whether the label is $e$ or $r(e)$ depends on the identifications made between normal fibers of 1-handle
and normal fibers of points in the boundary of a 0-handle.
We assume that that the normal fibers of points on the boundary of a single 0-handle are all given
trivializations which agree with some global orientation of the boundary of the 0-handle.)
If all the points on the boundary of a 0-handle are labeled by $e$ (or all are labeled by $r(e)$), then
the evaluation is $a$, as in the oriented and unoriented trivial $r$-action cases.
If there are points labeled by both $e$ and $r(e)$, then the evaluation for that 0-handle is zero (because
$e$ and $r(e)$ are orthogonal).
If $W$ is nonorientable, then for each labeling there will always be at least one 0-handle with mixed $e$ and $r(e)$ labeled points.
If $W$ is orientable, then there are precisely two labelings (corresponding to the two possible orientations of $W$)
for which all 0-handle evaluations are non-zero.
So we have
\[
	Z(W) = \begin{cases}
		2 a^{\chi(W)}\quad\quad\mbox{if $W$ is orientable} \\
		0\quad\quad\mbox{if $W$ is nonorientable}
	\end{cases}
\]

\noop{
\subsubsection{Spin}

Now consider $H = \Spin(1) \cong \z/2$, and enriched in super vector spaces.
An indecomposable $H$-pivotal 1-category is Morita equivalent to either the trivial category $\c$ or $\cl(1)$.
(Recall that as a super vector space $\cl(1) \cong \c^{1/1}$, and that the 
odd part is spanned by an element $s$ satisfying $s^2 = \id$.)
For the bar structure on $\cl(1)$, the requirement that

\bigskip
\nn{...}

}

\subsection{Dijkgraaf-Witten}
\label{ss-dw}

For simplicity I'll consider only untwisted Dijkgraaf-Witten theory, and also assume that $n \ge 2$.

Let $G$ be a finite group and consider the input $n$-category $\pi_{\le n}(BG)$, where $BG$ is the classifying space of $G$.
The $k$-morphisms of $\pi_{\le n}(BG)$ are maps of $k$-balls into $BG$ (if $k<n$) or
finite linear combinations of homotopy classes of maps of $n$-balls (with specified fixed boundary) into $BG$
(when $k=n$).

Note that $\pi_{\le n}(BG)$ is an $O(n)$-pivotal $n$-category, so our input manifold $W$ can be unoriented.

Since $n \ge 2$, $A(S^n)$ is 1-dimensional, and we choose the evaluation which sends the element of $A(S^2)$ represented by the 
trivial map $S^n\to BG$ to 1.

Since $\pi_k(BG)$ is trivial for $k \ne 1$, and $\pi_1(BG) \cong G$, it follows that labelings of cells are as follows.
\begin{itemize}
\item $n$-cells are labeled by points of $BG$, of which there is only one up to categorical equivalence.
We can take the 0-cell labels to be the standard base point * in BG.
\item $n{-}1$-cells are labeled by elements of $G$, or more specifically by choices (for each $g\in G$) of paths in $BG$,
from * to *, representing the homotopy class corresponding to $G$.
\item Each $n{-}2$-cell is adjacent to a cyclically ordered set of $n{-}1$-cells, with labels $g_1,\ldots,g_m$.
If $\prod g_i = 1$, then there is a unique (up to equivalence) 2-morphism in $\pi_{\le n}(BG)$ with the specified boundary, and the 
$n{-}2$-cell is labeled accordingly.
If $\prod g_i \ne 1$, then there is no 2-morphism with the specified boundary, and it is not possible to complete the 
partial labeling of $n$- and $n{-}1$-cells to a full labeling of all cells.
\item For $k$-cells, $k \le n-3$, there is always a unique (up to equivalence) of extending the labeling to the $k$-cell.
\end{itemize}
In summary, the set of labelings bijects with the set of maps from the $n{-}1$-cells to $G$ such that the product of group elements
adjacent to each $n{-}2$-cell is 1.

The handle-boundary evaluations are all equal to 1.

Applying the inductive definition of the norm $N(\beta(h))$ yields $N(\beta(h)) = 1$ if $h$ is a $(k<n)$-handle and
$N(\beta(h)) = |G|$ if $h$ is an $n$-handle.

Putting it all together, we have (for closed $W$)
\begin{eqnarray*}
Z(W) &=& 
			\sum_{\beta\in\cL(\cH)} \;\; \prod_{j=0}^{n+1}\;\;  \prod_{h\,\in\, \text{$j$-handles}} 
			\frac{\ev(\beta(\bd h))}{N(\beta(h))}
			\\
&=&			\sum_{\beta\in\cL(\cH)}
			\prod_{h_n} \frac{1}{|G|}
			\\
&=&			|\cL(\cH)| \prod_{h_n} \frac{1}{|G|}
\end{eqnarray*}
In other words, the the number of permissible labelings, with a factor of $1/|G|$ for each $n$-cell.
This is the usual (untwisted) Dijkgraaf-Witten state sum (for the Poincar\'e dual cell decomposition).

\subsection{Rep(G)}
\label{ss-repg}

Let $G$ again be a finite group, and let $\Rep(G)$ be its category of finite-dimensional representations.
We can think of $\Rep(G)$ as an $n$-category for any $n$.
The $k$-morphisms are trivial for $k < n-1$.
The $(n{-}1)$-morphisms are (roughly) objects of $\Rep(G)$.
The $n$-morphisms are (roughly) morphisms of $\Rep(G)$.
The corresponding string diagrams are ribbon graphs in $n$-dimensional manifolds, with ribbons labeled
by objects of $\Rep(G)$ and vertices labeled by the elements of the morphism space in $\Rep(G)$ corresponding to the 
incident ribbons.

(This construction works with $\Rep(G)$ replaced by any symmetric monoidal ribbon category, though in general one might
need to enrich in super vector spaces (instead of ordinary vector spaces) and use spin manifolds (rather than
oriented manifolds).)

The $\Rep(G)$ $n$-category is Morita equivalent to $\pi_{\le n}(BG)$ of the previous section, so the $\Rep(G)$ state
sum will compute the same invariant of closed $n{+}1$-manifolds (the untwisted Dijkgraaf-Witten invariant), but the details
of the state sum more closely resemble the Turaev-Viro and Crane-Yetter state sums.

The labelings $\beta\in\cL(\cH)$ are as follows.
\begin{itemize}
\item $k$-cells have trivial label if $k>2$.
\item 2-cells are labeled by simple objects $\rho$ in $\Rep(G)$ (i.e.\ irreducible representations of $G$).
\item 1-cells are labeled by an orthogonal basis the vector space of $G$-morphisms
$1 \to \rho_1 \ot \cdots\ot \rho_m$, where the $\rho_i$ are the simple
objects (or duals thereof, depending on relative 1-cell/2-cell orientations) assigned by the labeling to the 2-cells
incident to the 1-cell.
\end{itemize}

The handle-boundary evaluations $\ev(\beta(\bd h))$ are as follows.
(For simplicity we assume that the empty ribbon graph in $S^n$ evaluates to 1, rather than some general $\lambda\in \c$.)
\begin{itemize}
\item On the boundary of a $k$-handle, for $k>2$, we see the empty string diagram, which evaluates to 1.
\item The boundary of a 2-handle $h$ is an $n$-sphere with a single loop labeled by an irrep (simple object) $\rho = \beta(h)$.
The evaluation is the ordinary dimension (loop value) of $\rho$.
\item The boundary of a 1-handle $h$ is a ``mutant theta" graph with a labeled edge for each adjacent 2-handle and two vertices,
labeled by $\beta(h)$ and $\trev{\beta(h)}$.
\item The boundary of a 0-handle is a labeled ribbon graph which depends on the 1-skeleton of the link of the 0-handle and on the
labels assigned to the incident 1- and 2-handles.
\end{itemize}

The norms of morphisms are as follows.
\begin{itemize}
\item The norm of a 1-handle label $\alpha$ is $\ev(\alpha \cup \ol\alpha)^2 = (\Theta^m_\alpha)^2$,
where $\Theta^m_\alpha$ denotes a ``mutant" theta symbol.
\item The norm of a 2-handle label (simple object) is 1.
\item The norm of a (trivial) 3-handle label is $\sum \dim(\rho_i)^2 = |G|$.
(This is the TQFT evaluated on $S^2\times B^{n-1}$ with empty boundary conditions.)
\item The norm of a $k$-handle label, for $k \ge 3$, is $|G|^{k-1}$.
\end{itemize}

Putting it all together, we have (for closed $W$)
\begin{eqnarray*}
Z(W) &=& 
			\sum_{\beta\in\cL(\cH)} \;\; \prod_{j=0}^{n+1}\;\;  \prod_{h\,\in\, \text{$j$-handles}} 
			\frac{\ev(\beta(\bd h))}{N(\beta(h))}
			\\
&=&			
			\sum_{\beta\in\cL(\cH)} \;\; \prod_{j=3}^{n+1}\;\;  \prod_{h_j\,\in\, \text{$j$-handles}} 
			|G|^{j-1} 
			\prod_{h_2\,\in\, \text{2-handles}} \dim(\beta(h_2))
		\\ && \quad\quad\quad\quad
			\prod_{h_1\,\in\, \text{1-handles}} \frac{1}{\Theta^m_{\beta(h_1)}}
			\prod_{h_0\,\in\, \text{0-handles}} \ev(\Link(h_0))
\end{eqnarray*}

\section{Proof of invariance}
\label{s-np1}

This section contains the proof that the state sum of \ref{ss-sum} is independent of the choice of cell decomposition,
and that furthermore it is the top-dimensional part of an $n{+}1$-dimensional TQFT.
Subsection \ref{ss-pi2ss} gives some definitions and then shows that the state sum formula
follows easily from the TQFT gluing formula for $n{+}1$-manifolds.

Subsection \ref{ss-cpi} proves that the $n{+}1$-dimensional part of the TQFT exists.
Instead of working with triangulations (or their Poincar\'e dual cell decompositions), we work with arbitrary
handle decompositions.
This makes the invariance proof easier, since the ``moves" relating difference handle decompositions
(handle slides and handle cancellations) are simpler that the moves relating different triangulations
(Pachner moves).

\subsection{From path integral to state sum}
\label{ss-pi2ss}

Let $C$ be an $n$-category satisfying the hypotheses of \ref{n-cat}.
Then we can construct an $n{+}\epsilon$-dimensional TQFT, as outlined in Appendix \ref{s-nep}.
In particular, for each $n$-manifold $M$ we have the vector space $A(M)$ of string diagrams modulo local relations, and its
dual space $Z(M) = A(M)^*$.
A {\it path integral} is defined to be an element
\[
	Z(W)\in Z(\bd W) ,
\]
or equivalently, a function
\[
	Z(W) : A(\bd W) \to \k ,
\]
defined for each $n{+}1$-manifold $W$, satisfying the invariance and gluing conditions below.

{\bf Invariance.}
Let $F: W\to W'$ be an isomorphism of $n{+}1$-dimensional $H$-manifolds, and let 
$f: \bd W\to \bd W'$ denote the restriction of $F$ to boundaries.
Then
\[	
	f_*(Z(W)) = Z(W') .
\]

{\bf Gluing.}
Let $W$ be an $n{+}1$-manifold equipped with a decomposition of its boundary as 
$\bd W = M \cup \trev M \cup N$.
Let $W_{\gl}$ be the manifold obtained by gluing $M$ to $\trev M$ (see Figure \ref{gl-fig1}).
\kfigb{gl-fig1}{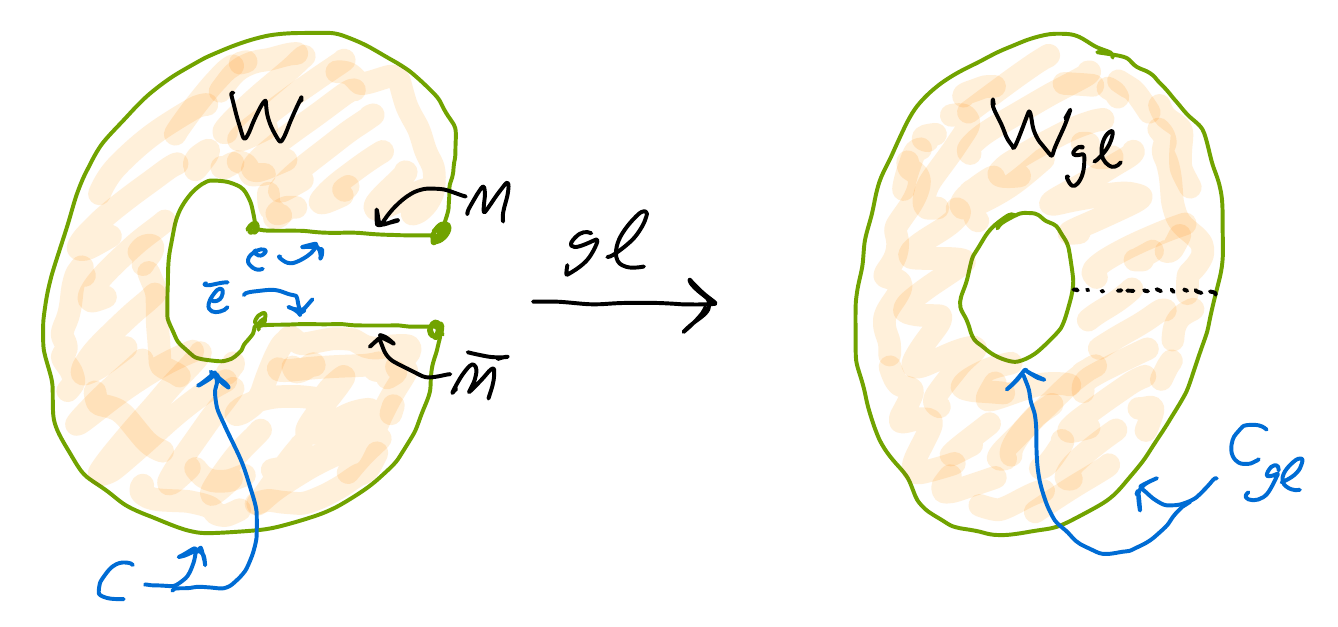}{Dramatis personae for gluing.}{5in}
Let $c \in A(N; x, \ol x)$ and let $c_{\gl}$ denote the corresponding glued-up string diagram in $\bd W_{\gl}$.
We want to express $Z(W_{\gl})(c_{\gl})$ in terms of $Z(W)$.
The only reasonable answer is to use the ``trace" map
\[
	Z(\bd W) \to \prod_b Z(M; b) \ot Z(\trev M, \trev b) \ot Z(N; \ol b \du b)
			\to \prod_b Z(N; \trev b \du b) \supset Z(\bd W_{\gl}) .
\]
The product is over string diagrams $b$ on $\bd M$.
The components of the first map are restrictions, and the second map is a trace with respect to the pairings determined by
\begin{eqnarray*}
	A(M;b) \ot A(\trev M; \trev b) & \to & \k \\
	x \ot \trev y & \mapsto &  Z(M\times I)(x \cup \trev y)  = \ip{x}{\trev y}.
\end{eqnarray*}
These pairings are assumed to be nondegenerate, and so determine pairings on the dual spaces
\[
	Z(M; b) \ot Z(\trev M, \trev b) \to \k .
\]

We will use a more concrete form of the gluing relationship
\[
	Z(W_{gl})(c_{gl}) = \sum_e \frac{Z(W)(c\cup e\cup\ol e)}{\ip{e}{\trev e}} ,
\]
where the sum is over an orthogonal basis of $A(M; b)$.
The factor of $\ip{e}{\trev e}$ reflects that fact that we are tracing out in the dual space $Z(M; b) = A(M; b)^*$
rather than $A(M; b)$.

For notational convenience, we will also use the isomorphism (potentially anti-linear) between
$A(M; b)$ and $A(\trev M; \trev b)$ to define pairings
\begin{eqnarray*}
	A(M;b) \ot A(M; b) & \to & \k \\
	x \ot y & \mapsto &  Z(M\times I)(x \cup \trev y)  = \ip{x}{y}.
\end{eqnarray*}

\medskip

\begin{thm}\label{pithm}
Let $C$ be an $H$-pivotal $n$-category such that
\begin{itemize}
	\item $A(S^k\times B^{n-k}; c)$ is finite-dimensional for all $k$ and for all boundary conditions $c$,
	\item $A(S^k\times B^{n-k-1}; c)$  
	is a semisimple 1-category for all $k$ and for all boundary conditions $c$, and
	\item there exists $z\in Z(S^n)$ which induces nondegenerate inner products
	on $A(B^n; c)$ for all boundary conditions $c$.
\end{itemize}
Then there exists a unique path integral $Z(\cdot)$, satisfying the invariance and gluing conditions described above, such
that
\[
	Z(B^{n+1}) = z .
\]
\end{thm}

The proof of \ref{pithm} is in Subsection \ref{ss-cpi}.

We will usually refer to $z$ in the above theorem as the evaluation map (evaluating a string 
diagram of $\bd B^{n+1}$ and producing a scalar), and write $\ev(x) = z(x)$.

\medskip

In the remainder of this subsection we will derive the state sum formula from path integral gluing relation.
As a warm-up, we will first do this in the $n=2$ case (obtaining the Turaev-Viro state sum).
Then we will do the general case.

\medskip

Let $n=2$ and, for simplicity, let $C$ be a fusion category.
We will assume that $Z(B^{n+1})(\eset) = \ev(\eset) = 1$.
As a preliminary, we need to compute inner products for the attaching regions of 1- 2- and 3-handles
($B^2\times S^0$, $B^1\times S^1$ and $B^0\times S^2$).

Inner products on $A(B^2; c)$ are given directly by the evaluation map $Z(B^3)$.
For a 3-valent 2-morphisms $\alpha$ and $\beta$, we have
\[
	\ip{\alpha}{\beta} = \Theta_{\alpha\beta} ,
\]
(see Figure \ref{a4}).
\kfigb{a4}{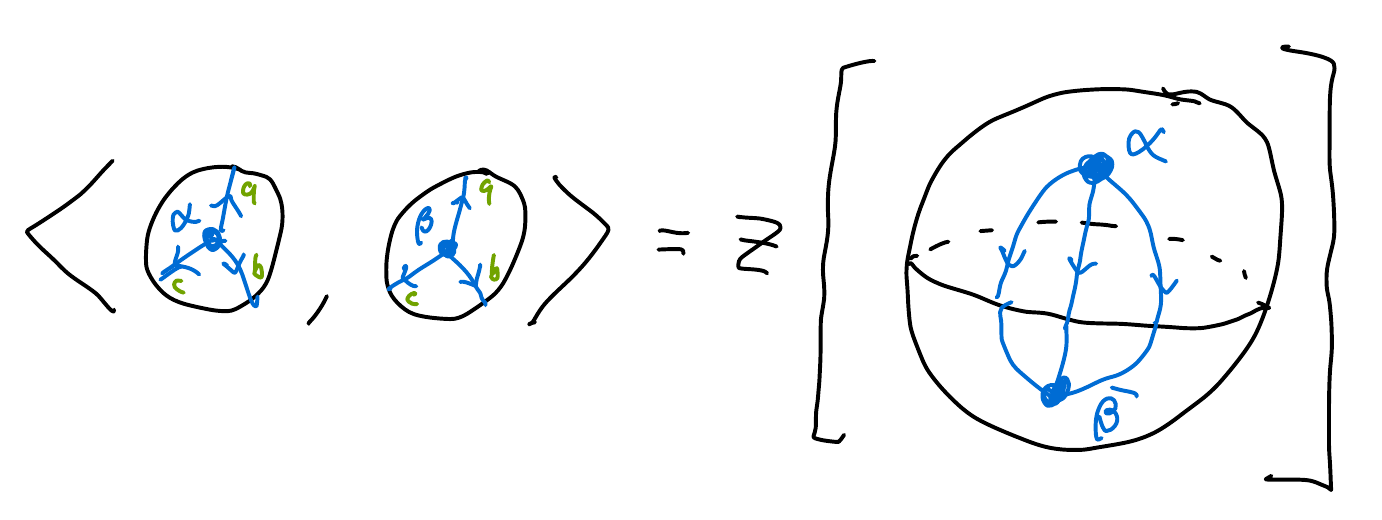}{Inner product definition.}{5in}
For a $k$-valent 2-morphism, the inner product is the evaluation of a mutant theta graph with $k$ edges
and two vertices labeled by $\alpha$ and $\trev\beta$.
A special case is
\[
	\ipp{\id_a} = d_a,
\]
where $d_a$ denotes the loop value or quantum dimension of $a$.

A basis of $A(B^1\times S^1; \eset)$ is given by $\{l_a\}$, where $l_a$ denotes the string diagram in $A(B^1\times S^1; \eset)$
consisting of a single loop $pt\times S^1$, and $a$ runs though simple objects (up to isomorphism).
By definition, we have
\[
	\ip{l_a}{l_b} = Z(B^2\times S^1)(l_{\trev a} \cup l_b) .
\]
To evaluate $Z(B^2\times S^1)(l_{\trev a} \cup l_b)$, we use the gluing formula
\[
	Z(B^2\times S^1)(l_{\trev a} \cup l_b) = \sum_e \frac{Z(B^2\times I)(\id_{\trev a} \cup \id_b\cup e \cup \trev e)}{\ipp e} .
\]
The sum is indexed by an orthogonal basis of $A(B^2; \trev a, b)$.
If $a\ne b$, then this space is 0-dimensional and the path integral is zero.
If $a=b$, this it is 1-dimensional, spanned by $\id_a$, and both the numerator and the denominator are equal to $d_a$.
Therefore
\[
	\ip{l_a}{l_b} = \delta_{ab} .
\]

A basis of $A(S^2)$ is given by the empty string diagram $\eset_{S^2}$.
We have
\begin{align*}
\ipp{\eset_{S^2}} & = Z(S^2\times I)(\trev{\eset_{S^2}} \cup \eset_{S^2}) \\
	& = \sum_{a} \frac{Z(B^2\times I)(\trev{l_a}) \cdot Z(B^2\times I)(l_a)}{\ipp{l_a}} \\
	& = \sum_a d_a^2 \\
	& = \gdim(C) .
\end{align*}

\medskip

With the above inner product calculations out of the way, we can now compute the path integral of
a 3-manifold in terms of a handle decomposition.
Let $M$
be a closed 3-manifold with handle decomposition $\cH$.
Let $M_i$ denote the union of the 0- through $i$-handles (so $M_3 = M$ and $M_i$ is obtained from 
$M_{i-1}$ by adding $i$-handles).

We will apply the gluing relation to express $M_i$ in terms of $M_{i-1}$, for $i=3,2,1$.
We know how to compute the path integral of $M_0$, as well as each $i$-handle, 
since we started out knowing $Z(B^3)$.
Assembling these results yields the desired state sum.

Attaching a 3-handle gives a factor of $1/\gdim(C)$:
\begin{align*}
	Z(M_3) & = Z(M_2)(\eset) \cdot \prod_{h\in\cH_3}  \frac{Z(h)(\eset_{S^2})}{\ipp{\eset_{S^2}}} \\
		& = Z(M_2)(\eset) \cdot \prod_{h\in\cH_3}  \frac{1}{\gdim(C)} .
\end{align*}

Attaching a 2-handle gives a factor of $d_a$:
\begin{align*}
	Z(M_2)(\eset) & = \sum_{\beta_2} Z(M_1)(L_{\beta_2}) \cdot \prod_{h\in\cH_2}  
					\frac{Z(h)(l_{\beta_2(h)} \cup \eset_{D^2} \cup \eset_{D^2})}{\ipp{l_{\beta_2(h)}}} \\
		& = \sum_{\beta_2} Z(M_1)(L_{\beta_2})\cdot \prod_{h\in\cH_2}  d_{\beta_2(h)} .
\end{align*}
The sum is over all labelings $\beta_2$ of 2-handles $h$ by simple objects $\beta_2(h)$.
The corresponding basis element of $A(S^1\times B^1)$ (where $S^1\times B^1$ is the attaching region of the 2-handle)
is $l_{\beta_2(h)}$.
The full boundary of the 2-handle is $(S^1\times B^1) \cup D^2 \cup D^2$, and $l_{\beta_2(h)} \cup \eset_{D^2} \cup \eset_{D^2}$
is the string diagram we see in the full boundary of the 2-handle.
(The $\eset_{D^2}$ is a portion of the string diagram we chose for the boundary of a 3-handle in the previous step.)
$L_{\beta_2}$ denotes the string diagram we see in the boundary of $M_1$.
There is a loop labeled by $\beta_2(h)$ for each 2-handle $h$.

Attaching a 1-handle gives a factor of $\Theta_\alpha\inv$:
\begin{align*}
	Z(M_1)(L_{\beta_2}) & = \sum_{\beta_1} Z(M_0)(G_{\beta_2, \beta_1}) \cdot \prod_{h\in\cH_1}  
					\frac{Z(h)(\beta_1(h) \cup \trev{\beta_1(h)} \cup (\bd \beta_1(h))\times I)}{\ipp{\beta_1(h)}^2} \\
		& = \sum_{\beta_1} Z(M_0)(G_{\beta_2, \beta_1})\cdot \prod_{h\in\cH_1}  \frac{\Theta_{\beta_1(h)}}{\Theta_{\beta_1(h)}^2} \\
		& = \sum_{\beta_1} Z(M_0)(G_{\beta_2, \beta_1})\cdot \prod_{h\in\cH_1}  \Theta_{\beta_1(h)}\inv .
\end{align*}
The sum is over all labelings $\beta_1$ of 1-handles $h$ by orthogonal basis vectors
$\beta_1(h) \in \mor(\tunit\to a\ot b\ot c)$,
where $a$, $b$ and $c$ are the simple objects (or their duals, depending on relative orientations of 1- and 2-handles)
assigned by $\beta_2$ to the three 2-handles adjacent to $h$.
$G_{\beta_2, \beta_1}$ is the disjoint union of tetrahedral graphs (one for each 0-handle) labeled by $\beta_2$ and $\beta_1$
applied to the 2- and 1-handles adjacent to each 0-handle.
The boundary of each 1-handle $h$ has the string diagram $\beta_1(h) \cup \trev{\beta_1(h)} \cup (\bd \beta_1(h))\times I$.
The $(\bd \beta_1(h))\times I$ part of the string diagram is on the non-attaching boundary 
and consists of three arcs, one for each adjacent 2-handle.

Since $M_0$ is a disjoint union of 3-balls, we have
\[
	Z(M_0)(G_{\beta_2, \beta_1}) = \prod_{h\in \cH_0} \text{Tet}(h) ,
\]
where $\text{Tet}(h)$ denotes the evaluation of the tetrahedral string diagram on the boundary of $h$.
The labels of the string diagram are determined by applying $\beta_2$ and $\beta_1$ to the 2- and 1-handles adjacent to $h$.

Putting it all together, we have
\[
	Z(M) = \sum_{\beta_2} \sum_{\beta_1} 
				\prod_{h\in\cH_3}  \gdim(C)\inv
				\prod_{h\in\cH_2}  d_{\beta_2(h)}
				\prod_{h\in\cH_1}  \Theta_{\beta_1(h)}\inv 
				\prod_{h\in \cH_0} \text{Tet}(h) .
\]
As discussed in \ref{ss-tv}, this is the specialization of the general state sum formula to the case where the input 
$n$-category is a fusion category ($n=2$).

\bigskip

We now consider the general case.
This is very similar to the $n=2$ warm-up case above, with one new ingredient:
we need to determine an orthogonal basis of $A(S^{k}\times B^{n-k}; c)$, the attaching region of a $k{+}1$-handle.
As we will see, a convenient basis is given by (equivalence classes of) minimal $k$-morphisms with appropriate boundary.

Let $W$ be a closed $n{+}1$-manifold with handle decomposition $\cH$.
As before, let $W_i$ be the union of the 0- through $i$-handles.
We will consider the problem of expressing $Z(W_i)$ in terms of $Z(W_{i-1})$ for $i=n+1$, $n$ and $n-1$, then do the general case.

$W_{n+1}$ is obtained from $W_{n}$ by gluing $n{+}1$-handles along $n$-spheres, so we must determine an orthogonal basis
of $A(S^n)$.
Any string diagram on $S^n$ is isotopic to one in which all of $S^n$ outside of a small ball $D$ is labeled by some
0-morphism $p\in C^0$.
This, in turn, is equivalent to a diagram where $S^n \setmin pt$ is labeled by $p$ and the point is labeled
by an element of $\End(\id^{n-1}p)$.
If $p$ is minimal, then $\End(\id^{n-1}p)$ is 1-dimensional and the diagram is a scalar multiple of the ``empty" string diagram
$\eset_p$ where all of $S^n$ is labeled by $p$.
If $p$ is not minimal, then (by our minimality assumption), $p$ is isomorphic to $\oplus p_j$, with each $p_j$ minimal.
It follows that $\eset_p$ is equal, in $A(S^n)$, to a linear combination of the $\eset_{p_j}$.
Thus $A(S^n)$ is spanned by string diagrams of the form $\eset_p$, with $p$ minimal.

Recall that minimal $p,q\in C^0$ are defined to be equivalent if there is a nonzero element of $C^1_{pq}$, 
a non-zero 1-morphism connecting $p$ to $q$.
It follows that if $p$ is equivalent to $q$, then $\eset_p$ is a non-zero scalar multiple of $\eset_q$ in $A(S^n)$.
On the other hand, if $p$ is not equivalent to $q$, then $\eset_p$ are $\eset_q$ are orthogonal.
Applying the gluing relation to $S^n\times I$, we see that
\[
	\ip{\eset_p}{\eset_q} = Z(S^n\times I)(\trev{\eset_p}\cup\eset_q) = \sum_e 
				\frac{Z(B^n\times I)(\cdots) \cdot Z(B^n\times I)(\cdots)}{\ipp{e}} .
\]
The sum is over an orthogonal basis of $A(S^{n-1}\times I; p,q) = A(S^{n-1}\times I; p\times S^{n-1} \cup q\times S^{n-1})$, 
but if $p$ and $q$ are not equivalent, then
this vector space is 0-dimensional and the inner product $\ip{\eset_p}{\eset_q}$ is zero.

We can now apply the gluing relation to obtain
\[
	Z(W_{n+1}) = \sum_\beta Z(W_n)(L_\beta) \cdot \prod_{h\in \cH_{n+1}} \frac{\ev(\beta(h)\times S^n)}{\ipp{\beta(h)\times S^n}} .
\]
The sum is over all labelings $\beta$ of $n{+}1$-handles $h$ by minimal 0-morphisms $\beta(h)$.
$L_\beta$ denotes the string diagram on $\bd W_n$ determined by $\beta$; 
it places $\eset_{\beta(h)} = \beta(h)\times S^n$ on the boundary component
corresponding to $h$.

\medskip

$W_{n}$ is obtained from $W_{n-1}$ by gluing $n$-handles along copies of $S^{n-1}\times B^1$,
with boundary condition $p\times S^{n-1}\cup q\times S^{n-1}$, 
where $p$ and $q$ are the minimal 0-morphisms assigned by $\beta$ to the
two $n{+}1$-handles adjacent to the $n$-handle $h$ that we are considering.
So we must determine an orthogonal basis
of $A(S^{n-1}\times I; p,q)$.
Let $a \in C^1_{pq}$ be a non-zero 1-morphism connecting $p$ and $q$, and 
let $a$ also denote the corresponding string diagram on $I$.
Any string diagram on $(S^{n-1}\times I; p,q)$ is equivalent to one which 
coincides with $S^{n-1}\times  a$ outside of a small ball, for some $a$.
This, in turn, is equivalent to a diagram where $S^{n-1}\times  a$ is decorated with a point on the $a$-labeled
$S^{n-1}$ labeled 
by an element of $\End(\id^{n-2}a)$.
If $a$ is minimal, then $\End(\id^{n-2}a)$ is 1-dimensional and the diagram is a scalar multiple of $S^{n-1}\times  a$.
If $a$ is not minimal, then (by our minimality assumption), $a$ is isomorphic to $\oplus a_j$, with each $a_j$ minimal.
It follows that $S^{n-1}\times  a$ is equal, in $A(S^{n-1}\times I)$, to a linear combination of the $S^{n-1}\times a_j$.
Thus $A(S^n)$ is spanned by string diagrams of the form $S^{n-1}\times a$, with $a$ minimal.

As before, we see that $S^{n-1}\times a$ is a non-zero scalar multiple of $S^{n-1}\times b$ if $a$ and $b$ are equivalent
minimal 1-morphisms, and $S^{n-1}\times a$ and $S^{n-1}\times b$ are orthogonal if $a$ and $b$ are non-equivalent mimimal morphisms.
So for our orthogonal basis, we can take $\{S^{n-1}\times a\}$, where $a$ runs though a set of representatives of
the equivalence classes of minimal 1-morphisms in $C^1_{pq}$.

We can now apply the gluing relation to obtain
\[
	Z(W_{n})(L_\beta) = \sum_{\gamma} Z(W_{n-1})(L_{\beta\gamma}) \cdot 
			\prod_{h\in \cH_{n}} \frac{\ev(\gamma(h)\times S^{n-1} \cup \bd\gamma(h)\times B^n)}{\ipp{\gamma(h)\times S^{n-1}}} .
\]
The sum is over all labelings $\gamma$ of $n$-handles $h$ by minimal 1-morphisms $\gamma(h)$.
$L_{\beta\gamma}$ denotes the string diagram on $\bd W_{n-1}$ determined by $\beta$ and $\gamma$.

\medskip

Expressing $Z(W_{n-1})(L_{\beta\gamma})$ in terms of $Z(W_{n-2})$ proceeds similarly.
The attaching region of an $n{-}1$-handle is $S^{n-2}\times B^2$.
The boundary of the attaching region is $S^{n-2}\times S^1$, and one should think of the $S^1$ factor 
as a linking circle for $n{-}1$-cell corresponding to the $n{-}1$-handle.
The labelings $\beta$ and $\gamma$ determine a string diagram $a$ on the linking circle.
The boundary condition for the attaching region is $a\times S^{n-2}$.
We seek an orthogonal basis of $A(S^{n-2}\times B^2; S^{n-2}\times a)$.
By the same argument as before, this is given by $\{S^{n-2}\times b\}$, where $b$ runs through minimal
2-morphisms with boundary $a$.
It follows that
\[
	Z(W_{n-1})(L_{\beta\gamma}) = \sum_{\delta} Z(W_{n-2})(L_{\beta\gamma\delta}) \cdot 
			\prod_{h\in \cH_{n-1}} \frac{\ev(\delta(h)\times S^{n-2} \cup \bd\delta(h)\times B^{n-1})}
						{\ipp{\delta(h)\times S^{n-2}}} .
\]
The sum is over all labelings $\delta$ of $n{-}1$-handles $h$ by minimal 2-morphisms $\delta(h)$.
$L_{\beta\gamma\delta}$ denotes the string diagram on $\bd W_{n-2}$ determined by $\beta$ ,$\gamma$ and $\delta$.

\medskip

The general case is very similar to the above special cases.
(Because of the finite size of the greek alphabet, we will rename $\beta = \beta_{n+1}$, $\gamma = \beta_n$, and
$\delta = \beta_{n-1}$.)
The attaching region of an $i$-handle is $S^{i-1}\times B^{n-i+1}$.
The boundary of the attaching region is $S^{i-1}\times S^{n-i}$, and one should think of the $S^{n-i}$ factor 
as a linking sphere for the $i$-cell corresponding to the $i$-handle.
The labelings chosen for higher index handles determine a string diagram $a$ on the linking sphere.
The boundary condition for the attaching region is $a\times S^{n-i}$.
We want an orthogonal basis of $A(S^{i-1}\times B^{n-i+1}; S^{n-i}\times a)$.
This is given by $\{S^{n-i}\times b\}$, where $b$ runs through minimal
$(n{-}i{+}1)$-morphisms with boundary $a$.
It follows that
\[
	Z(W_{i})(L_{\beta_i}) = \sum_{\beta_{i}} Z(W_{i})(L_{\beta_{i}}) \cdot 
			\prod_{h\in \cH_{i}} \frac{\ev(\beta_{i}(h)\times S^{n-i} \cup \bd\beta_{i}(h)\times B^{n-i+1})}
						{\ipp{\beta_{i}(h)\times S^{n-i}}} .
\]
The sum is over all labelings $\beta_{i}$ of $i$-handles $h$ by minimal $n{-}i$-morphisms $\beta_{i}(h)$.
$L_{\beta_{i+1}}$ denotes the string diagram on $\bd W_{i}$ determined by previous choices of labelings, and 
$L_{\beta_{i}}$, a string diagram on $\bd W_{i-1}$, is defined similarly.

Combining the above, for $i = n+1,n,\ldots,1$, yields the general state sum formula.
We combine the various $\beta_i$ 
into a single labeling $\beta$ of handles of all indices.
We can abbreviate
\[
	\ev(\beta(\bd h)) \deq \ev(\beta_{i}(h)\times S^{n-i} \cup \bd\beta_{i}(h)\times B^{n-i+1}) .
\]
The inductive definition of $N(\cdot)$ mirrors the inductive computation of inner products
on $A(S^{i-1}\times B^{n-i+1}; S^{n-i}\times a)$.
\noop{\nn{should probably discuss this separately and in more depth; or not?}}
\[
	N(\beta(h)) = \ipp{\beta_{i}(h)\times S^{n-i}} .
\]
Putting it all together, we obtain the desired state sum formula:
\[
	Z(W) = \sum_{\beta\in\cL(\cH)} \;\; \prod_{j=0}^{n+1}\;\;  \prod_{h\,\in\, \cH_j} 
			\frac{\ev(\beta(\bd h))}{N(\beta(h))}.
\]
When $W$ has boundary, there is a similar state sum formula for $Z(W)(x)$, where $x\in A(\bd W)$.

\subsection{Constructing the path integral}
\label{ss-cpi}

This section contains the proof of Theorem \ref{pithm}.
The proof given here is essentially the same as the one given in \cite{W2006}.

As remarked in the introduction,
in recent joint work with David Reutter \cite{RWnss} the inductive construction of the path integral
given below
has been generalized to non-semisimple contexts using less clunky techniques.
Algebraists and category theorists will likely prefer the new, more general proof.
But the older, less fancy proof presented here might appeal to more to low-dimensional topologists mainly interested 
in the semi-simple case.

\medskip

Define an $(m,i)$-handlebody to be an $m$-dimensional manifold equipped with a handle decomposition with all handle
indices less than or equal to $i$.
If $X$ is an $(m,i)$-handlebody, $X\times I$ will denote the $(m+1, i)$-handlebody obtained by thickening all the handles of $X$.

We will sometimes use the same symbol to denote an $(m,i)$-handlebody and the underlying manifold.
\noop{\nn{is this still true?}}
At other times, we will emphasize the distinction between the handlebody structure and the underlying manifold
by using the notation $\hb{M}$ to denote a handlebody whose underlying manifold is $M$.
The particular choice of handlebody structure will be clear from context.

We define two $(m,i)$-handlebodies to be equivalent if they are related by a series of handle slides and handle cancellations in
which all handle indices are less than or equal to $i$.
It is a standard result that two $(m,m)$-handlebodies are equivalent if and only if the underlying manifolds are
diffeomorphic/PL-homeomorphic.
\noop{But if $i<m$ \nn{...}}

\medskip

Here's a brief outline of the proof.
There are three steps.
First we show how to compute $Z(\hb W)$, in terms of the gluings encoded in the handlebody structure of $\hb W$,
using only $Z(B^{n+1})$ and the two axioms for the path integral.
(This is very similar to \ref{ss-pi2ss}.)
This shows that if the path integral exists it is unique.
Next we show that the result of the computation of step 1 depends only on $W$ and not the choice of handlebody $\hb W$.
This shows that we have a well-defined element $Z(W)\in Z(\bd W)$ for each $n{+}1$-manifold $W$.
Finally we show that the $Z(W)$ that we have thus defined does, indeed, satisfy the path integral axioms.

The first step proceeds by induction on the handle index.
(This is still part of the outline; more details will be given below.)
\begin{itemize}
\item $Z(B^{n+1})$ determines a non-degenerate pairing on $A(B^n)$,
\item which gives a recipe for computing $Z(\hb{S^1\times B^n})$ (by attaching 1-handle to $B^{n+1}$),
\item which determines a non-degenerate pairing on $A(S^1\times B^{n-1})$,
\item which gives a recipe for computing $Z(\hb{S^2\times B^{n-1}})$ (by attaching 2-handle to $B^{n+1}$),
\item $\ldots$
\end{itemize}
The finite-dimensionality and semisimplicity assumptions are used to prove non-degenerateness.

The second step boils down to showing that the recipe for attaching cancelling $k$ and $k{+}1$-handles yields the 
identity map.
The semisimplicity assumption is again used here.

The third step is an induction on the number and indexes of handles for a handle decomposition 
of the $n$-manifold we are gluing along.

\medskip

Now for the details of step 1.
By assumption, $z = Z(B^{n+1})$ determines a non-degenerate pairing on $A(B^n; c)$ for all boundary conditions $c$.
\eqar{
	A(B^n; c) \ot A(\ol{B^n}; \ol c) & \to  & \k \\
	x \ot \ol y & \mapsto & Z(B^{n+1})(x \cup \ol y)
}

As in Subsection \ref{ss-pi2ss}, we can use the above pairing to compute $Z(\hb{S^1\times B^n})$,
for the standard handlebody structure consisting of one 0-handle and one 1-handle.
(In that subsection we considered a more restricted set of boundary conditions for the path integral,
but the argument is the same for general boundary conditions.)

It's worth emphasizing that during this step 1 of the proof, $Z(\hb W)$ means applying the gluing formula to the sequence of gluings
specified by the handlebody structure $\hb W$.
At this stage we have not yet shown that $Z(\hb W)$ is independent of the choice of $\hb W$.

We first need to show that $Z(\hb{S^1\times B^n})$ is a well-defined function from $A(\bd(S^1\times B^n))\to \k$.

\begin{lem}
The gluing relation of \ref{ss-pi2ss} defines an element of $Z(\bd W_{gl})$.
In other words, if two string diagrams $c_{gl}$ and $c'_{gl}$ represent the same element of $A(\bd W_{gl})$, then
the values computed by the gluing relation for $Z(W_{gl})(c_{gl})$ and $Z(W_{gl})(c'_{gl})$ are equal.
\end{lem}

\begin{proof}
$A(\bd W_{gl})$ is a quotient of string diagrams modulo local relations.
Those relations are generated by (1) local relations supported away from the cut locus, and (2) an isotopy (non-local) which
shifts the cut locus.
It's clear that the lemma holds if $c_{gl}$ and $c'_{gl}$ are related by a local relation supported away from the cut locus,
so all that remains is to show that
\[
	Z(W_{gl})(\gl(c\bullet e)) = Z(W_{gl})(\gl(e\bullet c))
\]
for all $c$ and $e$ (see Figure \ref{a5}).
\kfigb{a5}{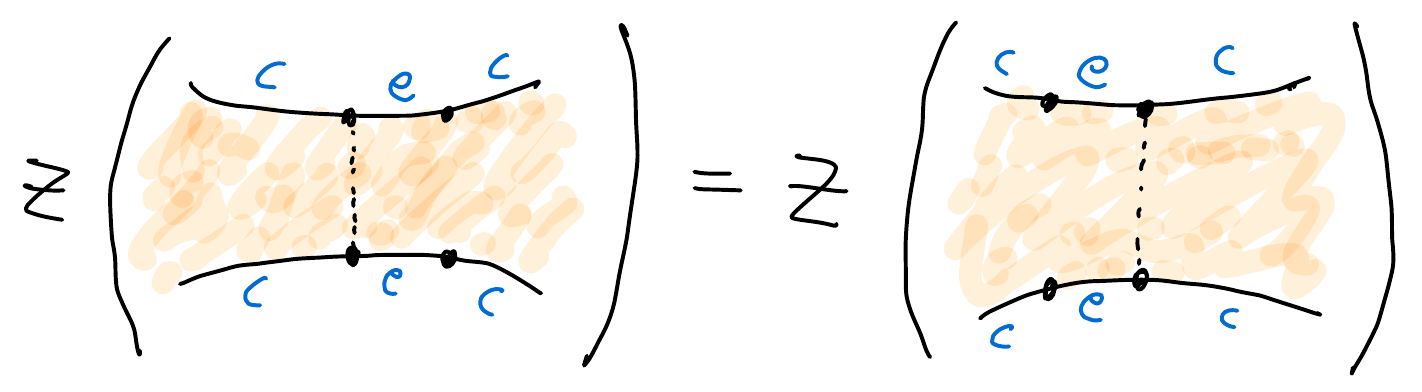}{Shift isotopy invariance.}{5.5in}

An alternate way of obtaining $W_{gl}$ is to glue $W$ to $M\times I$ (Figure \ref{a6}).
\kfigb{a6}{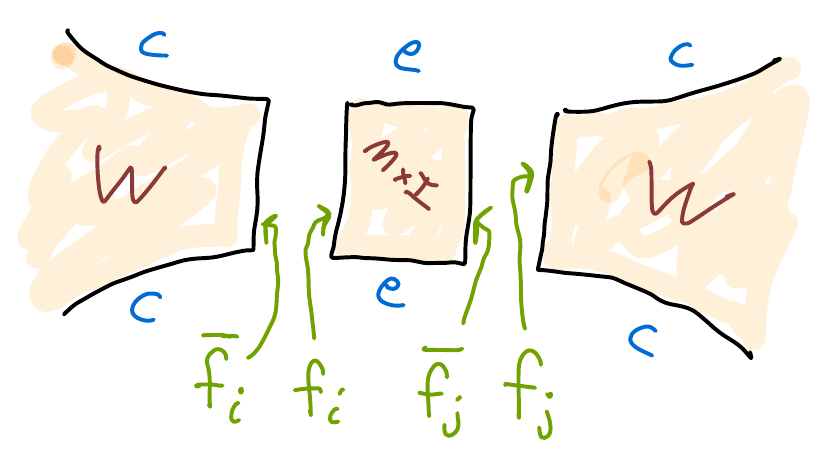}{More gluing.}{4.5in}
Applying the gluing relation to this decomposition, we obtain
\[
	Z(W_{gl})(\gl(c, e)) = \sum_{i,j} \frac{ Z(W)(c\cup \ol{f_i}\cup f_j) \cdot Z(M\times I)(e\cup f_i\cup \ol{f_j}) }
				{ \ipp{f_i} \ipp{f_j} }. \quad\quad\quad\quad(*)
\]
Using the facts that
\[
	Z(M\times I)(e\cup f_i\cup \ol{f_j}) = \ip{\ol{f_i}}{e\cup \ol{f_j}}
\]
and
\[
	e\cup \ol{f_j} = \sum_{i} \frac{ \ip{\ol{f_i}}{e\cup \ol{f_j}} }{ \ipp{\ol{f_i}} } \ol{f_i}
			= \sum_{i} \frac{ \ip{\ol{f_i}}{e\cup \ol{f_j}} }{ \ipp{f_i} }  \ol{f_i}
\]
and the linearity of $Z(X)(\cdot)$,
we see that (*) above is equal to
\[
	\sum_{j} \frac{ Z(W)(c\bullet e\cup \ol{f_j}\cup f_j)}{ \ipp{f_j} } ,
\]
which is what the gluing relation spits out for $Z(W_{gl})(\gl(c\bullet e))$.
Interchanging the roles of $i$ and $j$, we see that (*) is also equal to 
the gluing formula for $Z(W_{gl})(\gl(e\bullet c))$.
\end{proof}

We can now define a pairing for $A(S^1\times B^{n-1}; c)$ for all boundary conditions $c$.
\eqar{
	A(S^1\times B^{n-1}; c) \ot A(\ol{S^1\times B^{n-1}}; \ol c) & \to  & \k \\
	x \ot \ol y & \mapsto & Z(\hb{S^1\times B^n})(x \cup \ol y)
}
Here $\hb{S^1\times B^n}$ denotes the standard handlebody structure on $S^1\times B^n$, consisting of a 0-handle
and a 1-handle.
We must show that the above pairing is non-degenerate.

First we recall a standard skein theory result.
Let $M = M_1 \cup_Y M_2$ be an $n$-manifold decomposed into two pieces along $Y$.
Let $c$ be a boundary condition on $M$ and let $c'$ be the restriction of $c$ to $\bd Y$.
Assume that the 1-category $A(Y; c')$ is semisimple, and that $\{e_i : x_i\to x_i\}$ 
is a complete set of minimal idempotents for $A(Y; c')$.
Let $c_j$ denote the restriction of $c$ to $\bd M_j\cap \bd M$.
Then for each $i$ and $j$ the idempotent $e_i$ determines a subspace $A(M_j; c_j \cup e_i)$ of $A(M_j; c_j \cup x_i)$.
(Strictly speaking we should write $\ol{x_i}$ and $\ol{e_i}$ when $j=2$.
And even more strictly speaking, we should distinguish between different directions of bar-ing
for higher codimension manifolds.)
Then (this is the standard result)
\[
	A(M_1 \cup_Y M_2; c) \cong \bigoplus_i A(M_1; c_1 \cup e_i) \ot A(M_2; c_2 \cup e_i) .
\]
(In the super case, we would need to tensor over the endomorphism algebra of $e_i$ rather than over scalars.)

Next, we compute how the TQFT pairings interact with the above decomposition.
\begin{lem} \label{ipscalelemma}
With notation as above, let $a_j\in A(M_j; c_j \cup e_i)$ and $b_j \in A(\ol{M_j}; c_j \cup e_i)$.
Then
\[
	\ip{a_1\ot a_2}{b_1\ot b_2}_{A(M)} = \frac{ \ip{a_1}{b_1}_{A(M_1)} \ip{a_2}{b_2}_{A(M_2)} }{ \ip{e_i}{e_i}_{A(Y\times I)} }
\]
Let $b'_j \in A(\ol{M_j}; c_j \cup e_m)$, with $m \ne i$.  Then
\[
	\ip{a_1\ot a_2}{b'_1\ot b'_2}_{A(M)} = 0 .
\]
In other words, the above decomposition of $A(M; c)$ is orthogonal with respect to the TQFT pairings, and there are some
$\ip{e_i}{e_i}_{A(Y\times I)}$ scaling factors involved in relating the glued and unglued pairings.
\end{lem}
\begin{proof}
Since the pairings are defined in terms of path integrals, 
and we have not yet shown that well-defined path integrals exist, we need to
clarify the statement of the lemma.
The pairings make use of candidate values of $Z(M\times I)$, $Z(M_1\times I)$, $Z(M_2\times I)$, 
and $Z(Y\times I \times I)$, 
which are based on particular choices of handlebody structure on these manifolds,
and we assume that these candidate values satisfy the path integral gluing axiom for gluing $M_1\times I$ to 
$M_2\times I$ along $Y\times I$ to
obtain $M\times I$.
With this stipulation in place, the lemma follows immediately from the path integral gluing axiom.
(When we apply the lemma, $M_1$ will be a 0-handle and $M_2$ with be a $k$-handle, and the 
stipulation will be satisfied.)
\end{proof}

We can now show that the pairings for $A(S^1\times B^{n-1}; c)$ (for all $c$) are non-degenerate.
We have assumed that the 1-category $A(S^0\times B^{n-1}; c')$ is semisimple for all $c'$.
The above lemma now shows that the pairing for $A(S^1\times B^{n-1}; c)$ is an orthogonal sum of pairings, and each summand pairing
is a product of pairings for $A(B^n; d)$, which are assumed to be non-degenerate.
(The non-zero-ness of the scaling factors $\ip{e_i}{e_i}_{A(Y\times I)}$ from the lemma also follows from the 
non-degenerateness of the $B^n$ pairings.)

Proceeding inductively, we can show that the pairings on $A(S^k\times B^{n-k}; c)$ are non-degenerate for all $k$ and all $c$.
Assume this has been done for $0, \ldots, k-1$.
The standard handle decomposition of $S^k\times B^{n+1-k}$ attaches a $k$-handle to a 0-handle along $S^{k-1}\times B^{n-k+1}$.
The non-degeneracy of the pairings for $A(S^{k-1}\times B^{n-k+1}; d)$ allows us to compute $Z(S^k\times B^{n+1-k})$ in terms of
this handle decomposition.
We use $Z(S^k\times B^{n+1-k})$ to define a pairing on $A(S^k\times B^{n-k}; c)$.
We must show that this pairing on $A(S^k\times B^{n-k}; c)$ is non-degenerate.
We cut $S^k\times B^{n-k}$ into two copies of $B^n$ along $S^{k-1}\times B^{n-k}$.
The 1-categories $A(S^{k-1}\times B^{n-k}; c')$ are (by assumption) semisimple for all $c'$, so we can apply Lemma \ref{ipscalelemma}
to conclude that the pairings for $A(S^k\times B^{n-k})$ are non-degenerate (because the pairings for $A(B^n)$ 
and $A(S^{k-1}\times B^{n-k+1})$ are).

\medskip

This concludes step 1 of the proof.
Armed with the pairings for $A(S^k\times B^{n-k}; c)$ (for all $k$ and $c$) constructed above, we can compute the path integral
for any $n{+}1$-manifold equipped with an ordered handle decomposition.
(``Ordered" means that the handles are are attached sequentially in a specified order.
This order is not required to place lower-index handles before higher-index handles.)

\medskip

Step 2 of the proof is to show that these computations are independent of the choice of handle decomposition and
depend only on the underlying manifold.
Any two handle decompositions of a manifold are related by series of the following three ``moves":
\begin{itemize}
\item Swapping the order of a pair of distant handles which are adjacent in the order.
\item Handle slides (changing the attaching map of a handle by an isotopy).
\item Cancelling a $k$-handle and $k{+}1$-handle.
\end{itemize}

Invariance under distant order changes is obvious.
Invariance under handle slides is also obvious. 
(The theories are topologically invariant, so changing
a handle attaching map by an isotopy does not make any difference.)
All that remains for step 2 is to show that the computation of the path integral is invariant under handle cancellation.

Handle cancellation invariance will follow from the following associativity-of-gluing property 
of the path integral gluing formula.
To simplify notation we will ignore boundary conditions on $\bd W$.
\begin{lem}  \label{assoc-gluing}
Let $W = W_1 \cup W_2 \cup W_3$, with pairwise intersections $M_{ij} = W_i\cap W_j$, and common intersection an $n{-}1$-manifold $Y$.
See Figure \ref{b1}.
Assume that the 1-category $A(Y)$ is semisimple.
Then applying the gluing relation to first compute $W_1 \cup W_2$, and then to compute 
$(W_1 \cup W_2) \cup W_3$, yields the same answer as applying the 
gluing relation to $W_1 \cup W_3$, and then to $(W_1 \cup W_3) \cup W_2$.
In other wards, the path integral gluing formula is associative.
\end{lem}
\kfigb{b1}{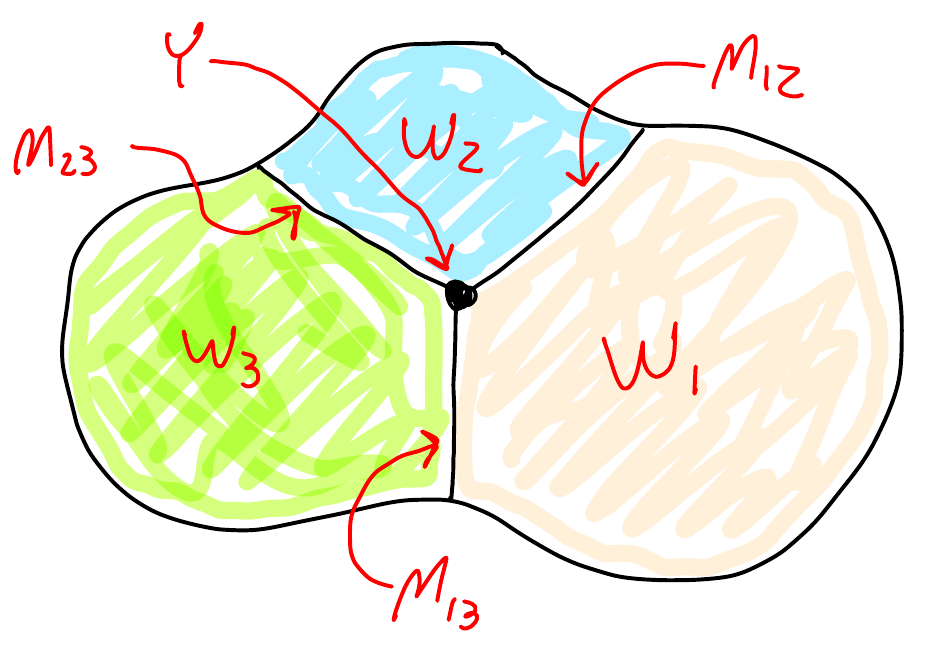}{Three manfolds glued together.}{4in}
\begin{proof}
Let $\{f_\alpha : x_\alpha\to x_\alpha\}$, with $\alpha\in J$, be a set of minimal idempotents for the semisimple 1-category $A(Y)$.
For each $\alpha$, $i$ and $j$, let $\{e_{\alpha\beta}\}$, 
with $\beta\in \Lambda_{ij\alpha}$, be an orthogonal basis of $A(M_{ij}; x_\alpha)$.
Then, by Lemma \ref{ipscalelemma}, $\{e_{\alpha\beta} \cup e_{\alpha\gamma}\}$,
with $\alpha\in J$, $\beta\in \Lambda_{13\alpha}$ and $\gamma\in \Lambda_{23\alpha}$,
is an orthogonal basis of $M_{13}\cup_Y M_{23}$, and similarly for permutations of 1,2,3.
See Figure \ref{b2}. 
\kfigb{b2}{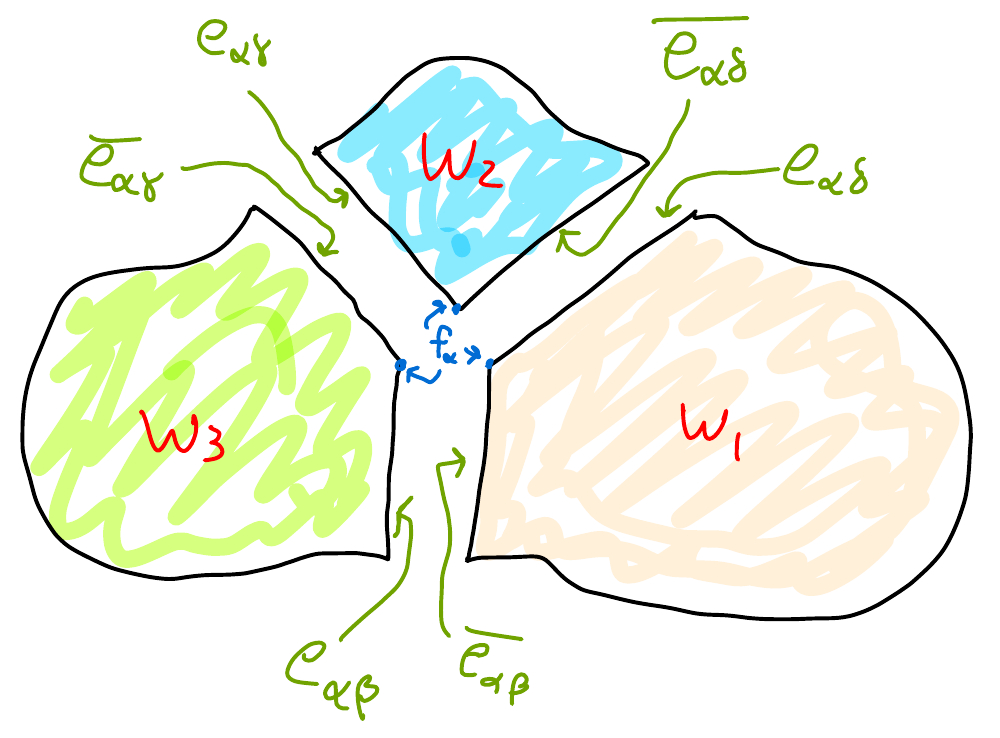}{Three manifolds cut apart.}{3.5in}
(To simplify notation, I'm omitting some bars (e.g.\ $\trev{e_{\alpha\beta}}$).)

Applying the gluing formula to compute $Z(W_1\cup W_2)$, we have
\[
	Z(W_1\cup W_2)(e_{\alpha\beta} \cup e_{\alpha\gamma}) = 
			\sum_{\delta\in \Lambda_{12\alpha}} 
				\frac{ Z(W_1)(e_{\alpha\beta} \cup e_{\alpha\delta})\cdot Z(W_2)(e_{\alpha\delta} \cup e_{\alpha\gamma}) }
						{ \langle e_{\alpha\delta}, e_{\alpha\delta} \rangle } .
\]
Applying the gluing formula again to compute $Z((W_1\cup W_2)\cup W_3)$, and then applying Lemma \ref{ipscalelemma}, we have
\begin{align*}
Z((W_1\cup W_2)\cup W_3) &=& \sum_{\alpha, \beta, \delta} 
		\frac{ Z(W_1)(e_{\alpha\beta} \cup e_{\alpha\delta})\cdot Z(W_2)(e_{\alpha\delta} \cup e_{\alpha\gamma})
									\cdot Z(W_3)(e_{\alpha\beta} \cup e_{\alpha\gamma})}
						{ \langle e_{\alpha\delta}, e_{\alpha\delta} \rangle 
						\langle e_{\alpha\beta} \cup e_{\alpha\gamma}, e_{\alpha\beta} \cup e_{\alpha\gamma} \rangle } \\
						&=& \sum_{\alpha, \beta, \delta} 
				\frac{ Z(W_1)(e_{\alpha\beta} \cup e_{\alpha\delta})\cdot Z(W_2)(e_{\alpha\delta} \cup e_{\alpha\gamma}) 
									\cdot Z(W_3)(e_{\alpha\beta} \cup e_{\alpha\gamma}) 
											\cdot \langle f_\alpha, f_\alpha \rangle}
						{ \langle e_{\alpha\delta}, e_{\alpha\delta} \rangle 
						\langle e_{\alpha\beta}, e_{\alpha\beta} \rangle
						\langle  e_{\alpha\gamma}, e_{\alpha\gamma}\rangle
						 }
\end{align*}
Note that the above expression is symmetric in permutations of 1,2,3.
It follow that applying the gluing formula to compute $Z((W_1\cup W_3)\cup W_2)$ (or $Z((W_2\cup W_3)\cup W_1)$) yields the same answer.
This completes the proof of the lemma, except for a footnote.

That footnote being: 
We want to apply the lemma in cases where the inner products on $A(M_{ij})$ and
$A(M_{ij}\cup M_{jk})$ (i.e.\  $Z(M_{ij}\times I)$ and $Z((M_{ij}\cup M_{jk}) \times I)$
are initially defined in terms of particular 
choices of handlebody structures on $M_{ij}$ and
$M_{ij}\cup M_{jk}$, and we must be careful to verify that Lemma \ref{ipscalelemma} holds for these choices.
\end{proof}

We will apply the lemma with $W_1$ the initial manifold, $W_2$ a $k$-handle, and $W_3$ a cancelling
$k{+}1$-handle.
See Figure \ref{a7}.
\kfigb{a7}{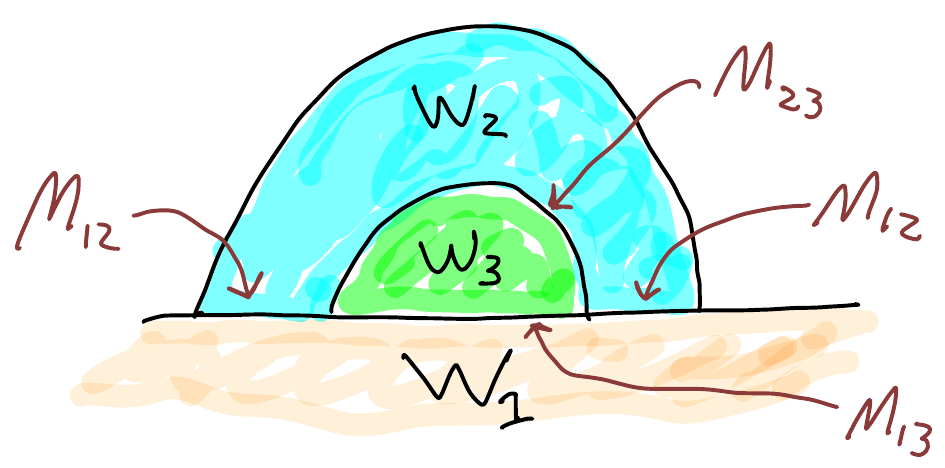}{Cancelling $k$- and $k{+}1$-handles.}{4in}

Then $M_{12}$ is $S^{k-1}\times B^{n-k+1}$ (the attaching region of the $k$-handle),
$M_{13}$ is $B^k\times B^{n-k}$ (half of the attaching region of the $k{+}1$-handle),
and $M_{23}$ is also $B^k\times B^{n-k}$ (the other half of the attaching region of the $k{+}1$-handle).
For $M_{12}\times I$ we choose the standard handle structure with one 0-handle and one $k{-}1$-handle.
For $M_{13}\times I$ and $M_{23}\times I$ we choose handle structures consisting of a single 0-handle.

$M_{12}\cup M_{23}$ is $B^n$, and for $(M_{12}\cup M_{23})\times I$ we consider two
handlebody structures: a single 0-handle, and a 0-handle plus a $k{-}1$-handle plus a cancelling $k$-handle.
Our inductive assumptions allow us to assume that these two handlebody structures yield the same result for 
$Z((M_{12}\cup M_{23})\times I)$.
(This is a key point.)
It follows that Lemma \ref{ipscalelemma} holds for $M_{12}\cup M_{23})$.

$M_{13}\cup M_{23}$ is $S^k\times B^{n-k}$ (the attaching region of the $k{+}1$-handle).
We choose the standard handlebody structure on $(M_{13}\cup M_{23})\times I$ (one 0-handle and one $k$-handle).
With these choices Lemma \ref{ipscalelemma} holds for $M_{13}\cup M_{23})$.

We can now apply Lemma \ref{assoc-gluing} to conclude that the computation of $Z((W_1\cup W_2)\cup W_3)$ agrees 
with the computation of $Z((W_1\cup W_3)\cup W_2)$.
In other words, we can attach the $k$- and $k{+}1$-handles in either order.
Our goal is to show that attaching the $k$-handle first and then the cancelling $k{+}1$-handle
is the same as doing nothing (i.e. path integral of $W_1$ and $(W_1\cup W_2)\cup W_3$ are the same).
Attaching the $k{+}1$-handle before the $k$-handle is equivalent to adding a boundary collar along $B^n$ (i.e.\ $M_{13}$).
Likewise, attaching the $k$-handle to the union of the original manifold and the $k{+}1$-handle is again equivalent to attaching
a boundary collar along $B^n$ (i.e.\ $M_{12}\cup M_{23}$).
It is easy to see that attaching boundary collars has no effect on the path integral.
It follows that the computation of a path integral by applying the gluing formula to a handlebody
structure is invariant under handle cancellation.

We have now shown that defining $Z(W)$ in terms of a choice of handlebody structure on $W$
is independent of the choice of handlebody structure.
Thus we have a well-defined path integral $Z(W) \in Z(\bd W)$ for every $n{+}1$-manifold $W$.
This completes step two of the proof of Theorem \ref{pithm}.

\medskip

The final step in the proof of Theorem \ref{pithm} is to show that the path integral $Z(W):A(\bd W)\to \k$ that we have just defined
does in fact satisfy the gluing formula, for any gluing of $n{+}1$-manifolds.

It suffices to show that the gluing formula holds for manifolds of the form $M = M' \cup h$, 
where $M'$ is an $n$-manifold for which we have already verified the gluing formula 
(i.e.\ verified the gluing formula for any gluing along $M'$), and $h$ is an $n$-dimensional $k$-handle.
(The induction starts with $M'$ empty and $h$ a 0-handle.)

Let $W$ be an $n{+}1$-manifold as in the statement if the gluing formula.
Let $W_{\gl}$ be the result of gluing $W$ to itself along $M$.
Let $W'_{\gl}$ be the result of gluing $W$ to itself along $M'$.
Let $H$ be a thickened version of $h$ -- an $n{+}1$-dimensional $k{+}1$-handle.
The boundary of $H$ can be divided into three pieces: $B^k\times B^{n-k}$ (``upper" attaching region),
another copy of $B^k\times B^{n-k}$ (``lower" attaching region),
and $B^{k+1}\times S^{n-k-1}$ (the non-attaching region).
See Figure \ref{a9}.
\kfigb{a8}{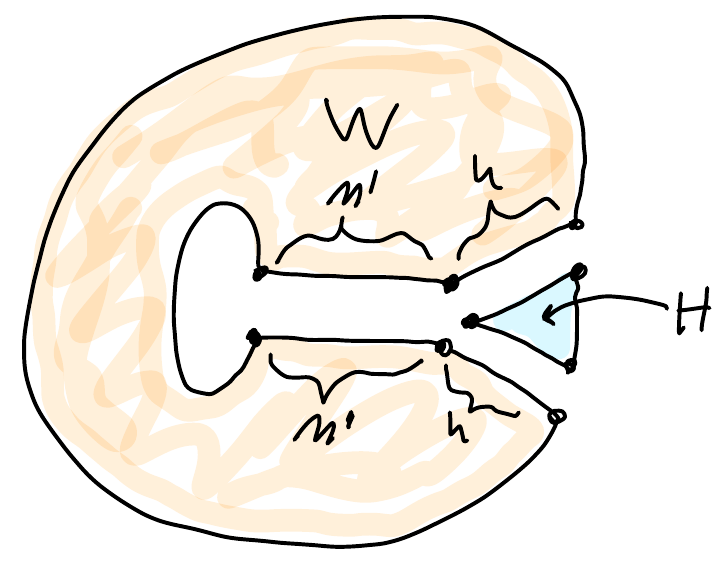}{Another figure.}{3.2in}
\kfigb{a9}{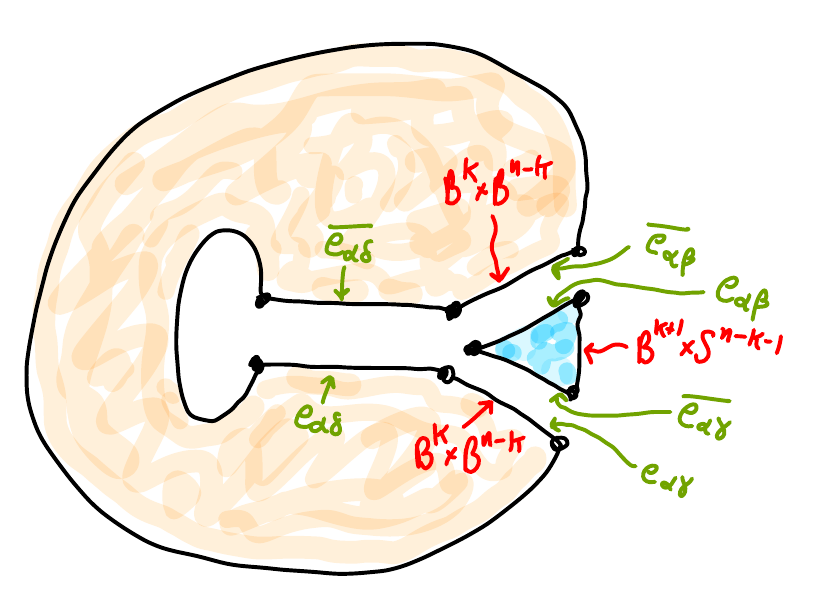}{And another.}{3.2in}

As usual, we will suppress from the notation boundary conditions on $\bd W_{\gl}$.

Note that $W_{\gl} \cong W'_{\gl} \cup H$.

As in the proof of Lemma \ref{assoc-gluing}, let $\{f_\alpha\}$ be a set of minimal idempotents of  $S^{k-1}\times B^{n-k+1}$
(the intersections of the upper and lower attaching regions), and
let $\{e_{\alpha\beta}\}$ and $\{e_{\alpha\gamma}\}$ be bases for the upper and lower attaching regions.
Since the upper and lower attaching regions are canonically isomorphic (or rather, bar-isomorphic), 
we can choose the ``same" basis for
each manifold and there is a natural bijection between these sets.
It follows that $\{e_{\alpha\beta}\cup \trev{e_{\alpha\gamma}}\}$ is a basis of the union of the upper and lower attaching regions,
with $\beta$ and $\gamma$ running through the same indexing set.

Since $Z(W_{\gl})$ can be computed with any handle decomposition, we can choose a handle decomposition which attaches
the $k{+}1$-handle $H$ last (see Figure \ref{a8}), 
and we have
\[
	Z(W_{\gl}) = \sum_{\alpha,\beta,\gamma} 
		\frac{Z(W'_{\gl})(\trev{e_{\alpha\beta}}\cup e_{\alpha\gamma}) \cdot Z(H)(e_{\alpha\beta}\cup \trev{e_{\alpha\gamma}})}
			{\langle e_{\alpha\beta}\cup \trev{e_{\alpha\gamma}}, e_{\alpha\beta}\cup \trev{e_{\alpha\gamma}} \rangle} .
\]
By our inductive hypotheses we have
\[
	Z(W'_{\gl})(\trev{e_{\alpha\beta}}\cup e_{\alpha\gamma}) = \sum_{\alpha,\delta} 
		\frac{Z(W)(e_{\alpha\delta}\cup \trev{e_{\alpha\delta}}\cup \trev{e_{\alpha\beta}}\cup e_{\alpha\gamma})}
			{\langle e_{\alpha\delta}, e_{\alpha\delta} \rangle} ,
\]
where $\{e_{\alpha\delta}\}$ is a basis of $A(M')$ (with boundary condition corresponding to $\alpha$).
By Lemma \ref{ipscalelemma}, we have
\[
	\langle e_{\alpha\beta}\cup \trev{e_{\alpha\gamma}}, e_{\alpha\beta}\cup \trev{e_{\alpha\gamma}} \rangle
		= \frac{\langle e_{\alpha\beta}, e_{\alpha\beta} \rangle \langle \trev{e_{\alpha\gamma}}, \trev{e_{\alpha\gamma}} \rangle}
			{\langle f_\alpha, f_\alpha \rangle}
\]
and
\[
	\langle e_{\alpha\delta}\cup e_{\alpha\gamma}, e_{\alpha\delta}\cup e_{\alpha\gamma} \rangle
		= \frac{\langle e_{\alpha\delta}, e_{\alpha\delta} \rangle \langle e_{\alpha\gamma}, e_{\alpha\gamma} \rangle}
			{\langle f_\alpha, f_\alpha \rangle}.
\]
(The last LHS is an inner product for $A(M) = A(M'\cup h)$.)
We also have
\[
	Z(H)(e_{\alpha\beta}\cup \trev{e_{\alpha\gamma}}) = \delta_{\beta\gamma} \langle e_{\alpha\beta}, e_{\alpha\beta} \rangle 
\]
(because $H$ is isomorphic to the product $n{+}1$-manifold used to define $\langle e_{\alpha\beta}, e_{\alpha\beta} \rangle$
and we have chosen orthogonal bases).

Combining all of the above, we have
\begin{align*}
	Z(W_{\gl}) &= \sum_{\alpha,\beta,\gamma,\delta}
		\frac{ Z(W)(e_{\alpha\delta}\cup \trev{e_{\alpha\delta}}\cup \trev{e_{\alpha\beta}}\cup e_{\alpha\gamma})   \delta_{\beta\gamma} \langle e_{\alpha\beta}, e_{\alpha\beta} \rangle  \langle f_\alpha, f_\alpha \rangle  }
			 { \langle e_{\alpha\delta}, e_{\alpha\delta} \rangle \langle e_{\alpha\beta}, e_{\alpha\beta} \rangle \langle \trev{e_{\alpha\gamma}}, \trev{e_{\alpha\gamma}} \rangle  } \\
&= \sum_{\alpha,\gamma,\delta}
		\frac{ Z(W)(e_{\alpha\delta}\cup \trev{e_{\alpha\delta}}\cup \trev{e_{\alpha\gamma}}\cup e_{\alpha\gamma})  \langle f_\alpha, f_\alpha \rangle  }
			 { \langle e_{\alpha\delta}, e_{\alpha\delta} \rangle \langle e_{\alpha\gamma}, e_{\alpha\gamma} \rangle  } \\
&= \sum_{\alpha,\gamma,\delta}
		\frac{ Z(W)(e_{\alpha\delta}\cup \trev{e_{\alpha\delta}}\cup \trev{e_{\alpha\gamma}}\cup e_{\alpha\gamma})   }
			 {\langle e_{\alpha\delta}\cup e_{\alpha\gamma}, e_{\alpha\delta}\cup e_{\alpha\gamma} \rangle }
\end{align*}
This is exactly the statement of the gluing formula for gluing $W$ along $M$, so we are done.

\medskip
A final remark:
The assumption that $C$ is weakly complete is not needed for the above path integral theorem.
The weakly complete assumption is only needed to write the state sum formula in a more compact form.

\noop{
\subsection{Temp notes}

\begin{itemize}
\item maybe remark that weakly complete (minimal morphism) assumptions is needed for simple ss formula, but
not needed for PIthm proof
\item super case: non-identity endomorphisms of minimal idems
\item do super case afterwards? (diffs); or use footnotes?
\item no, put Spin and Pin cases in a separate paper
\end{itemize}
}

\appendix

\section{Constructing the TQFT in dimensions $n$ through $0$}
\label{s-nep}

This appendix gives a terse account of how one constructs a fully extended $n{+}\epsilon$-dimensional
$H$-TQFT from an $\k$-linear $H$-pivotal $n$-category $C$.
The pivotality assumption is important here but $C$ need not satisfy any finiteness or semisimplicity
conditions.

In contrast to the $n{+}1$-dimensional path integral construction above, 
the constructions in this section are ``easy" in the sense that
there is no need to choose a combinatorial description of the manifolds and verify independence of that choice.

For more details, see \cite{W2006} and \cite{MWblob}.

Note that only the $n$- and $n{-}1$-dimensional parts of the TQFT are used in the rest of the paper.

\medskip

For $X$ a $k$-manifold, $0\le k \le n$, and $c$ a string diagram on $\bd X$, define $C(X; c)$ to be the set of all string
diagrams on $X$ which restrict to $c$ on $\bd X$.

Let $M$ be an $n$-manifold and let $c$ be a string diagram on $\bd M$ (i.e. $c\in C(\bd M)$).
Define $\k[C(M; c)]$ to be finite $\k$-linear combinations of such string diagrams on $M$.

Let $B$ be an $n$-ball (isomorphic to the standard $n$-ball $B^n$, but not necessarily canonically so).
Let $c$ be a string diagram on $\bd B$.
There is an evaluation map from $\k[C(B; c)]$ to a $\k$-vector space of $n$-morphisms of $C$.
(The domain/range of this space is determined by $c$.)
Let $U(B; c)\sub \k[C(B; c)]$ be the kernel of this evaluation map.

Now let $B\sub M$ be an $n$-ball contained in $M$.
For compatible string diagrams $c\in C(\bd M)$ and $d\in C(\bd B)$ and $e\in C(\bd (M \setmin B))$, there is a gluing map
\[
	\k[C(M\setmin B; e)] \ot U(B; d) \to \k[C(M; c)].
\]
Define $U(M; c)$ to be the span of the images of the above gluing maps, for all $B$, $d$, and $e$.
Finally, define
\[
	A(M; c) \deq \k[C(M; c)] / U(M; c) .
\]
We can think of $A(M; c)$ as finite linear combinations of string diagrams on $M$,
modulo the obvious local relations.
In other words, the $C$-skein module of $M$ (with boundary condition $c$).

We can also think of $A(M; c)$ is the pre-dual Hilbert space of the TQFT associated to $C$.
The Hilbert space is defined to be $Z(M;c) \deq A(M; c)^*$, functions on string diagrams with evaluate to zero on $U(M; c)$.
The path integral $Z(W^{n+1})$ constructed above is an element of $Z(\bd W)$.

\medskip

Now let $Y$ be an $n{-}1$-manifold and $c\in C(\bd Y)$.
We define a linear 1-category $A(Y;c)$ as follows.
The objects of $A(Y;c)$ are defined to be the string diagrams $C(Y;c)$.
(Note that we do not mod out by any relations.)
The morphisms from $x$ to $y$ are the vector space $A(Y\times I; \trev x \cup y)$.
(Note that we are using the ``pinched" boundary convention here, so that the entire boundary of 
$Y\times I$ is $\trev Y \cup Y$.)
Composition of morphisms is given by stacking/gluing.

We define $Z(Y;c)$ to be the representation category of $A(Y;c)$, 
i.e. functors from $A(Y;c)$ to Vec (the linear and additive completion trivial category).

Let $M$ be an $n$-manifold ($n$-dimensional $H$-manifold).
It is easy to see that the collection of vector spaces
$\{A(M;c)\}$, indexed by $c\in C(\bd M)$, affords a representation of $A(\bd M)$.
The action is given by gluing boundary collars onto $M$.
We will denote this representation by $A(M)$.

Let $M = M_1\cup_Y M_2$.
For simplicity assume that $Y$ is the entire boundary of both $M_1$ and $M_2$.
It is not hard to prove that 
\[
	A(M) \cong A(M_1)\ot_{A(Y)} A(M_2)
\]
(see \cite{W2006}).
More generally, there is a similar gluing theorem for self-gluings along non-closed $n{-}1$-manifolds $Y$.

\medskip

The $n$- and $n{-}1$-dimensional parts of the TQFT sketched above are all that is needed in this paper.
But it's not difficult to extend the above constructions all the way down to 0-manifolds.

Let $X$ be an $n{-}k$-manifold and $c\in C(\bd X)$.
We want to define a linear $k$-category $A(X; c)$.
For notational simplicity, I'll suppress some boundary conditions from the notation.
We define the $j$-morphisms of $A(X; c)$ (with $0\le j < k$) to be
the set of string diagrams $C(X\times B; \cdot)$, where $B$ is a $j$-ball.
We define the $k$-morphisms to be the vector space $A(X\times B; \cdot)$, where $B$ is a $k$-ball.

Since we know how to restrict string diagrams to boundaries and how to glue string diagrams together,
there are various domain/range and composition relationships among the above morphisms sets.
Whether it is now easy to show that we have constructed a linear $H$-pivotal $k$-category depends on the definition
of $k$-category one is using.
If one uses the disklike $n$-category definition of \cite{MWblob}, which is designed around exactly this example,
then showing that we have a $k$-category is easy.


\noop{  

\section{blah}
\label{xxxxsect}

\begin{itemize}
\item 
\end{itemize}

\bigskip \nn{...}

}       


\renewcommand*{\bibfont}{\small}
\setlength{\bibitemsep}{0pt}
\raggedright
\printbibliography

\end{document}